\documentclass[11pt]{amsart}

\def\flnm{hl}
\def\version{7}
\def\datum{June 2012}
\let\iflabels\iffalse
\let\ifcitenumber\iftrue
\let\iffilename\iffalse
\let\iftxfnts\iftrue
\let\ifscrltx\iffalse

\iflabels
\usepackage[notcite,notref]{showkeys}
\fi

\ifscrltx
\usepackage{srcltx}
\fi

\usepackage[all]{xy}

\iftxfnts
\usepackage[varg]{txfonts}
\fi

\newtheorem{thm}{Theorem}[section]
\newtheorem{prop}[thm]{Proposition}

\newtheorem{lem}[thm]{Lemma}

\theoremstyle{definition}
\newtheorem{defn}[thm]{Definition}

\newcommand\al{\alpha}
\newcommand\Bt{{\mathsf{B}}}

\newcommand\Gm{\Gamma}
\newcommand\gm{\gamma}
\newcommand\Dt{\Delta}
\newcommand\e{\varepsilon}
\newcommand\z{\zeta}
\newcommand\ld{\lambda}
\renewcommand\th{\vartheta}
\newcommand\x{\xi}
\newcommand\s{\sigma}
\newcommand\ph{\varphi}
\newcommand\ps{\psi}
\newcommand\Om{\Omega}
\newcommand\om{\omega}

\newcommand\CC{\mathbb{C}}
\newcommand\RR{\mathbb{R}}
\newcommand\ZZ{\mathbb{Z}}

\newcommand\cc{{\mathbf c}}
\newcommand\EE{\mathbf{E}}

\newcommand\Pt{{\mathcal P}}
\newcommand\PtY{{\mathcal P}_{\!Y}}

\renewcommand\={\;=\;}
\let\setminus\smallsetminus
\newcommand\be{\begin{equation}}
\newcommand\ee{\end{equation}}
\newcommand\bad{\be\begin{aligned}}
\newcommand\ead{\end{aligned}\ee}
\newcommand\badl[1]{\be\label{#1}\begin{aligned}}
\newcommand\eadl{\end{aligned}\ee}

\newcommand\re{\mathrm{Re}\,}
\newcommand\im{\mathrm{Im}\,}
\newcommand\SL{{\mathrm{SL}}}
\newcommand\PSL{{\mathrm{PSL}}}

\newcommand\uhp{{\mathfrak{H}}}
\newcommand\oh{{\mathrm{O}}}
\newcommand\Gf{\Gamma}
\newcommand\mer{{!!}}
\newcommand\ra{{\mathrm{an}}}
\newcommand\mf{\mathfrak{M}}
\newcommand\Four[1]{{\mathcal{F}_{\!#1}}}
\newcommand\four[1]{{\mathrm{F}_{\!#1}}}
\newcommand\sign{\mathrm{Sign}}
\newcommand\hypg[2]{{}_{#1}\!F_{\!#2}}
\newcommand\ha{\mathsf{H}}
\newcommand\co{\mathsf{C}}
\newcommand\mo{\mathsf{M}}
\newcommand\Gcom{\Gm_{\!\mathrm{com}}}

\makeatletter
\newcommand\hmatc[4]{\left[ {#1\@@atop #3}{#2\@@atop #4}\right]}
\newcommand\hmatr[4]{\left[ {\hfill #1\@@atop\hfill #3}{\hfill
#2\@@atop\hfill #4}\right]}
\newcommand\matc[4]{\left( {#1\@@atop #3}{#2\@@atop #4}\right)}
\newcommand\matr[4]{\left( {\hfill #1\@@atop\hfill #3}{\hfill
#2\@@atop\hfill #4}\right)} \makeatother

\iffilename
\makeatletter\def\@oddfoot{\rm{\footnotesize{File name
\tt\flnm\version.tex}\quad\today\hfil\thepage}}
\let\@evenfoot\@oddfoot 
\makeatother
\fi

\numberwithin{equation}{section}

\begin{document}

\title{Harmonic lifts of modular forms}

\author{Roelof Bruggeman}
\address{Mathematisch Instituut Universiteit Utrecht, Postbus 80010,
NL-3508 TA Utrecht, Nederland}
\email{r.w.bruggeman@uu.nl}

\date{\datum}

\begin{abstract}It is shown that each complex conjugate of a
meromorphic modular form for $\mathrm{SL}_2(\mathbb{Z})$ of any
complex weight $p$ occurs as the image of a harmonic modular form
under the operator $2i y^p \, \partial_{\bar z}$. These harmonic
lifts occur in holomorphic families with the weight as the parameter.
\end{abstract}
\maketitle

\section{Introduction} In the theory of mock modular forms, see \S3
of~\cite{BrFu} and also \S5 of~\cite{Za09}, one meets the exact
sequence
\be\label{es} 0 \longrightarrow M_p^! \longrightarrow H_p^! \stackrel
{\x_p} \longrightarrow \bar M_{2-p}^! \;, \ee
and gives conditions under which the last map is surjective. Here
$M_p^!$ denotes the space of holomorphic modular forms of weight~$p$
with at most exponential growth at the cusps
(also called the space of weakly holomorphic modular forms), and
$H_p^!$ denotes the corresponding space of $p$-harmonic modular
forms, defined by replacing the condition of holomorphy by the
condition of $p$-harmonicity, which means being in the kernel of the
operator
\be\label{Deltap} \Dt_p \= - 4 \,(\im z)^2 \,\partial_z\partial_{\bar
z}
+ 2 i p \,(\im z)\, \partial_{\bar
z}\,. \ee
The operator $\x_p = 2 i \,(\im z)^p\, \partial_{\bar z}$ maps $H_p^!$
into the space $\bar M_{2-p}^!$ of antiholomorphic modular forms of
weight~$2-p$ with at most exponential growth at the cusps. The
elements of $\bar M_{2-p}^!$ are complex conjugates of elements of an
appropriate space $M_{2-p}^!$. (In \S\ref{sect-mf} we will give a
more precise discussion of these spaces of modular forms.)

A \emph{$p$-harmonic lift} of an element $F$ in $\bar M^!_{2-p}$ is an
element $H$ of $H^!_p$ such that $\x_p H = F$. The concept stems from
the study of mock modular forms. Zwegers started in \cite{Zw02} with
mock theta functions $M$, which are holomorphic functions on the
upper half-plane given by a $q$-series, and added a simpler but
non-holomorphic function $C$ to it such that $M+C$ has modular
transformation behavior. The function $M+C$ is no longer holomorphic,
but $p$-harmonic for some weight~$p$. Applying the operator $\x_p$ to
$C$, or to $M+C$, gives an antiholomorphic cusp form of weight $2-p$,
from which $C$ can be reconstructed. Conversely, we may ask for a
given antiholomorphic automorphic form $F$ of weight~$2-p$ whether it
occurs as the image under $\x_p$ of a $p$-harmonic form~$H$.

Poincar\'e series form a convenient tool to construct harmonic lifts.
See Theorem~1.1 in the paper \cite{BO} of Bringmann and Ono, or \S6
in \cite{BOR} by Bruinier, Ono and Rhoades. If the parameters of the
Poincar\'e series are in the domain of absolute convergence this
gives a description of harmonic lifts by absolutely convergent
series. For other values of the parameters one has to use the
meromorphic continuation of the Poincar\'e series. An alternative
approach is the use of Hodge theory. See Corollary~3.8 in~\cite{BrFu}
of Bruinier and Funke. The method of holomorphic projection can be
used to construct harmonic lifts. See \S3 and \S5 in the
preprint~\cite{BKZ} of Bringmann, Kane and Zwegers.
\smallskip

My purpose in this paper is to show that the approach with Poincar\'e
series can be modified to work for arbitrary complex weights. I will
use results from perturbation theory of automorphic forms as
investigated in~\cite{Br94}. To avoid complications I consider only
the full modular group.
\begin{thm}\label{thmA}Let $F$ be an antiholomorphic modular form on
$\SL_2(\ZZ)$ of weight $2-p\in \CC$ with multiplier system $v$ on
$\SL_2(\ZZ)$ suitable for the weight~$p$, and assume that $F$ has at
most exponential growth at the cusps. Then there exists a
$p$-harmonic modular form $H$ on $\SL_2(\ZZ)$ of weight $p$ with the
same multiplier system $v$ and at most exponential growth at the
cusps, such that $\x_p H = F$.
\end{thm}

This is a mere existence result. The construction of $H$ is based on
the resolvents of self-adjoint families of operators in Hilbert
spaces, and does not give the $p$-harmonic lift $H$
explicitly.\smallskip

Let us denote by respectively $M_p^!(v)$, $\bar M_p^!(v)$ and
$H^!_p(v)$ the spaces of respectively holomorphic, antiholomorphic
and harmonic modular forms, with at most exponential growth, weight
$p$, and multiplier system~$v$.

Holomorphic and antiholomorphic modular forms occur in families, for
instance the powers of the Dedekind eta-function $r\mapsto \eta^{2r}$
form a family holomorphic in the weight $r\in \CC$, with a multiplier
system that we denote by~$v_r$. We have $\eta^{2r}\in M^!_r(v_r)$,
and $\bar \eta^{2r}\in \bar M^!_r(v_{-r})$. All antiholomorphic
modular forms with at most exponential growth are of the form
$F \,\bar\eta^{-2r}$, where $r\in \CC$, and $F\in
\bar M_{2-\ell}^!(1)$ for some $\ell\in 2\ZZ$. Such a family
$r\mapsto F\,\bar\eta^{-2r}$ is a holomorphic family on~$\CC$. It
turns out that harmonic lifts also occur in families, which however
are not defined on all of $\CC$, due to a branching phenomenon. We
work with domains of the form
\be U_M \= \CC \setminus [12 M,\infty) \ee
with $M\in \ZZ$.
\begin{thm}\label{thmB}Let $F\in \bar M^!_{2-\ell}(1)$ with
$\ell\in 2\ZZ$. There is $\mu_F\in \ZZ$ such that for all integers
$M\geq \mu_F$  there are holomorphic families
$r\mapsto \ha_{M,r}$ on $U_M$ for which
 $\ha_{M,r}\in H^!_{\ell+r}(v_r)$ and $\x_{\ell+r} \ha_{M,r} = F\,
\bar\eta^{-2r}$ for all $r\in U_M$.
\end{thm}
This result implies Theorem~\ref{thmA}.\medskip

Meromorphic modular forms may have singularities in points of the
upper half-plane~$\uhp$. The space $M^\mer_p(v_r)$ of meromorphic
modular forms of weight~$p$ with the multiplier system $v_r$ is
contained in the space $H^\mer_p(v_r)$ of harmonic functions $F$ on
$\uhp\setminus S$ that are invariant under the action of $\SL_2(\ZZ)$
of weight $p$ with the multiplier system $v_r$, where $S\subset\uhp$
consists of finitely many $\Gm$-orbits and where $F$ satisfies near
each $\z\in  S$ an estimate $F(z)
= \oh\Bigl( \bigl( \frac{z-\z}{z-\bar\z} \bigr)^{-a}\Bigr)$ as
$z\rightarrow \z$ for some $a>0$. The space $\bar M^\mer_p(v_r)$
consists of the complex conjugates of the functions in
$M^\mer_{\bar p}(v_{-\bar r})$.

\begin{thm}\label{thmC}Let $p,r\in \CC$ with $p\equiv r\bmod 2$. For
each $F\in \bar M^\mer_{2-p}(v_r)$ there exists a harmonic lift
$H\in H^\mer_p(v_r)$ such that $\x_p H = F$.
\end{thm}
This lifting can also be done in holomorphic families. Theorems
\ref{thmB} and~\ref{thmC} follow from the more general
Theorem~\ref{thm-main} in \S\ref{sect-ext}.\medskip

To obtain these results we start in \S\ref{sect-mf} with a more
precise discussion of the spaces of holomorphic, antiholomorphic and
harmonic modular forms. Section~\ref{sect-ramf} reformulates the
equation $\x_{\ell+r}H = \bar \eta^{-2r}F$ in terms of the more
general class of real-analytic modular forms. In this way we can
embed the family $r\mapsto \bar\eta^{-2r}F$ in a family with two
parameters, the weight and a ``spectral parameter''. This makes it
possible to use analytic perturbation theory to arrive at meromorphic
families $r\mapsto H_{N,r}$ of modular solutions of the equation
$\x_{\ell+r} H_{N,r} = \bar\eta^{-2r}F$. Section~\ref{sect-hfhf}
removes the singularities of these families, and leads to
Theorem~\ref{thm-main}, from which Theorems \ref{thmA}--\ref{thmC}
follow.

Section~\ref{sect-norm} gives a normalization that determines the
families of harmonic lifts uniquely. That does not mean that we
obtain them explicitly. The theorems in this paper are existence
results only. It is far from obvious how to write $h_{N,r}$ as the
sum of a ``mock modular form'' and a ``harmonic correction'',
especially if $h_{N,r}$ has singularities in the upper half-plane.
See~\S\ref{sect-mock}.

Subsection~\ref{generalization} discusses the possibilities and
difficulties of extension to other discrete groups. Finally,
Section~\ref{sect-etap} discusses, as an example, a lift of
$r\mapsto \bar\eta^{-2r}$, and states an explicit formula for the
first derivative of this lift at $r=0$.\bigskip

I thank Kathrin Bringmann and Ben Kane for several discussions on the
subject of this paper during several visits to Cologne. During the
symposium \emph{ Modular Forms, Mock Theta Functions, and
Applications} at Cologne in 2012, Soon-Yi Kang, \'Arp\'{a}d T\'{o}th
and Sander Zwegers discussed in their lectures methods to obtain
harmonic lifts. Several aspects of this paper are related to work in
progress with YoungJu Choie and Nikos Diamantis. I profited from
comments of Kathrin Bringmann, Jan Bruinier and Jens Funke on an
earlier version of this paper.

\section{Modular forms}\label{sect-mf}
This section serves to define the concepts more precisely than in the
introduction. The discrete group is $\Gm:=\SL_2(\ZZ)$.

\subsection{Holomorphic modular forms}\label{sect-holmf}
The Dedekind eta-function
\[ \eta(z) \= e^{\pi i z/12}\prod_{n\geq 1}\bigl(1-\nobreak e^{2\pi i
n z}\bigr)\]
has no zeros in the upper half-plane
$\uhp= \bigl\{z=x+iy\in \CC\;:\; y>0\bigr\}$. One chooses a branch of
its logarithm
\be \label{logeta}
\log\eta(z) \= \frac{\pi i z}{12} - \sum_{n\geq 1}\s_{-1}(n) \,
q^n\,,\ee
with $q=e^{2\pi i z}$ and $\s_u(n) = \sum_{d\mid n}d^u$, and then
defines $\eta^{2r}(z) \= e^{2 r \log\eta(z)}$. The transformation
behavior of $\log\eta$ is studied by R.\,Dedekind in the
appendix~\cite{Ded} to the collected works of B.\,Riemann. One may
also consult Chap.~IX in~\cite{La}. This leads to the modular
transformation behavior
\be \eta^{2r}(\gm z) = v_r(\gm)\, (cz+d)^r \, \eta(z)\qquad\text{for
all }\gm=\matc abcd\in \Gm\,,\ee
with the \emph{multiplier system}~$v_r$. A multiplier system suitable
for the weight $p\in \CC$ is a map $v:\Gm\rightarrow\CC^\ast$ such
that
\be\label{holact} \Bigl( F|_{v,r}\matc abcd \Bigr)\,(z)
\= v\matc abcd^{-1}\,
(cz+d)^{-p}\, F\Bigl( \frac{az+b}{cz+d}\Bigr)
\ee
defines a representation of $\Gm/\{\pm I\} = \PSL_2(\ZZ)$ in the
functions on~$\uhp$. We use the convention of computing
complex powers of $cz+d$ with $\arg(cz+\nobreak d) \in (-\pi,\pi]$.

For the modular group all multiplier systems occur in one family
$r\mapsto v_r$ with parameter $r\in \CC\bmod12\ZZ$. The multiplier
system $v_r$ is suitable for weights $p\equiv r\bmod 2$. It is
determined on the two standard generators of $\SL_2(\ZZ)$ by
\be v_r \matc 11{}1 \= e^{\pi i r/12}\,, \quad v_r\matr{}{-1}1{} \=
e^{-\pi i r/2}\,.\ee

\begin{defn}\label{def-M}Let $p,r\in \CC$, $p\equiv r\bmod 2$. The
space $M_p^!(r) = M_p^!\bigl(\Gm,v_r)$ consists of the holomorphic
functions $F$ on~$\uhp$ that satisfy $F|_{v_r,p}\gm = F$ for all
$\gm\in \Gm$ and
\be\label{expgr} F(z) \= \oh\bigl( e^{ay} \bigr) \quad\text{ as
}y\rightarrow \infty \text{ for some }A>0, \ee
uniformly for $x$ in compact sets. By $M^\mer_p(r)$ we denote the
space of meromorphic modular forms of weight $p$ with multiplier
system $v_r$.
\end{defn}
Here and in the sequel we use the standard convention $x=\re z$ and
$y=\im z$ for $z\in \uhp$. If $p$ and $r$ are not real we cannot
impose in \eqref{expgr} uniformity in~$x\in \RR$. The condition
\eqref{expgr} is the condition of \emph{exponential growth}.

Since $\eta$ has no zeros in $\uhp$, multiplication by $\eta^{2r_1}$
gives a bijection between $M^!_p(r) $ and
$M^!_{p+r_1}(r+\nobreak r_1)$. This implies that all spaces in
Definition~\ref{def-M} can be uniquely described as $M_{\ell+r}^!(r)$
with $r \in \CC$ and $\ell \in L:=\{0,4,6,8,10,14\}$. The general
form of an element of $M_{\ell+r}$ is $p(J) \, E_\ell \, \eta^{2r}$,
where $p(J)$ is a polynomial in the elliptic invariant
$J\in M_0^!(0)$, and $E_\ell$ is the holomorphic Eisenstein series in
weight $\ell\in L\setminus\{0\}$, and where we put $E_0=1$. The
general form of an element of $M^\mer_{\ell+r}(r)$ is also
$p(J)\, E_\ell\, \eta^{2r}$, where now $p(J)$ is a rational function
in~$J$. (See, {\sl e.g.}, \S4.1 of~\cite{Mi}.)

We can formulate the meromorphy of $F$ at $\z\in \uhp$ by holomorphy
in~$z$ on a pointed neighborhood of~$\z$ in~$\uhp$ and the growth
condition
\be\label{growth-int}
F(z) \= \oh\Bigl( \bigl( (z-\z)/(z-\bar\z) \bigr)^{-a}\Bigr)\qquad
\text{ as $z\rightarrow\z$, \ for some $a>0$}\,. \ee

\subsection{Harmonic modular forms and antiholomorphic modular
forms}\label{sect-hahmf}

\begin{defn}We say that a function $F$ is \emph{$p$-harmonic} on some
subset of~$\uhp$ if $\Dt_p F=0$ on that subset, where $\Dt_p$ is the
operator given in~\eqref{Deltap}.
\end{defn}
The holomorphic action $|_{v_r,p}$ of $\Gm$ in~\eqref{holact}
preserves $p$-harmonicity and commutes with $\Dt_p$. Holomorphic
functions are $p$-harmonic for each $p\in\CC$.

Antiholomorphy is not preserved by the action $|_{v_r,q}$ of $\Gm$,
but by the following action:
\begin{defn}For $p\equiv -r\bmod 2$ the antiholomorphic action
$|_{v_r,p}^a$ of $\Gm$ in the functions on~$\uhp$ is given by
\be \Bigl(F|_{v_r,p}^a \matc abcd\Bigr)(z) \= v_r \matc abcd^{-1}
(c\bar z+d)^{-p}\, F\Bigl( \frac {az+b}{cz+d}\Bigr)\qquad\text{ for }
\matc abcd\in \Gm\,, \ee
where powers of $c\bar z+d$ are computed with
$-\pi\leq \arg(c\bar z+\nobreak d)<\pi$.
\end{defn}

The operator
\be \x_p \= 2i y^p\,\partial_{\bar z} \ee
vanishes precisely on the holomorphic functions, and a function $F$ is
$p$-harmonic if and only if $\x_p F $ is antiholomorphic. This
operator $\x_p$ intertwines holomorphic and antiholomorphic actions:
\be \x_p\Bigl( F|_{v_r,p}\gm\Bigr) \= \bigl( \x_p
F\bigr)|^a_{v_r,2-p} \gm\,. \ee\medskip

Let us define $M^S_p(r)$ as the space of $F\in M^\mer_p(r)$ with
singularities contained in the set $S$. Then
$M^!_p(r) = M^\emptyset_p(r)$ and $M^\mer_p(r) = \bigcup_S M^S_p(r)$
where $S$ runs over the collection of unions of finitely many
$\Gm$-orbits in~$\uhp$. This suggests the following definition.

\begin{defn}\label{MHSdef}Let $S\subset \uhp$ consist of finitely many
(possibly zero)
$\Gm$-orbits in~$\uhp$. Let $p\equiv r\bmod 2$. We define $H^S_p(r)$
as the space of $r$-harmonic functions on $\uhp\setminus S$ that are
invariant under the action $|_{p,r}$ of~$\Gm$, satisfy the condition
\eqref{expgr} of exponential growth at the cusp and the growth
condition \eqref{growth-int} at the points $\z\in S$.

The space $\bar M^S_{-p}(r)$ consists of the antiholomorphic functions
on~$\uhp\setminus S$ that are invariant under the action
$|_{-p,v_r}^a$ of~$\Gm$ and satisfy \eqref{expgr}, and condition
\eqref{growth-int} for all $\z\in S$.

We put $H^!_p(r)=H^\emptyset_p(r)$, $\bar M_{-p}^!(r)= \bar
M^\emptyset_{-p}(r)$, $H^\mer_p(r) = \bigcup_S H^S_p(r)$, and $\bar
M^\mer_{-p}(r)= \bigcup_S \bar M^S_{-p}(r)$, where $S$ runs over the
collection of unions of finitely many $\Gm$-orbits in~$\uhp$.
\end{defn}
\[ \renewcommand\arraystretch{1.2}
\begin{array}{|l|c|c|c|}\hline
&M_p^S(r) & H_p^S(r) & \bar M^S_{-p}(r)\\ \hline
\text{growth at $\infty$:}& \multicolumn{3}{|c|}{\text{condition
\eqref{expgr} is satisfied}}\\
\text{near $\z\in S$:}&\multicolumn{3}{|c|}{\text{condition
\eqref{growth-int} is satisfied}}\\
\text{for all $\gm\in \Gm$:}& \multicolumn{2}{|c|}{\text{invariant
under $|_{v_r,p}\gm$}}& \text{invariant under $|_{v_r,-p}\gm$}\\
\text{on $\uhp\setminus S$:}
&\text{holomorphic}&\text{$p$-harmonic}&\text{antiholomorphic}\\
\hline
\end{array}
\]

For each $S$ consisting of finitely many $\Gm$-orbits in~$\uhp$ we
have an exact sequence
\be\label{es0}0 \rightarrow M_p^S(r) \rightarrow H^S_p(r)
\stackrel{\x_p}\rightarrow \bar M^S_{2-p}(r)\,.\ee
This is not immediately clear. The question is whether the operator
$\x_p$ preserves the growth conditions \eqref{expgr}
and~\eqref{growth-int}. This can be shown by looking at the growth of
the terms in the expansions in the next subsection, and the effect of
the operator~$\x_p$. It is a special case of an analogous result for
Maass forms, that we will mention in the next section. Near the end
of \S\ref{sect-exp-grc} we will derive the statement from
Proposition~4.5.3 in~\cite{Br94}.

The central question in this paper is whether
$\x_p H^S_p(r) \rightarrow \bar M_{2-p}^S(r)$ is surjective.

\subsection{Expansions} We fix a union $S$ of finitely many
$\Gm$-orbits in $\uhp$. Let $\Pt$ be the set consisting of $\infty$
and of representatives of the $\Gm$-orbits in~$S$, for instance 
representatives in the standard fundamental domain.

A meromorphic modular form $F\in M^S_p(r)$ has a Fourier expansion at
$\infty$ of the form
\be F(z) \= \sum_{\nu\geq \mu} a_\mu \, q^{\nu+r/12}\,,\ee
with $q^\al = e^{2\pi i \al z}$. The integer $\mu$ may  be
negative. If $F$ has singularities at points of~$S$ then this
expansion converges only on a region $y>A$ not intersecting~$S$.

Near each $\z\in \Pt\cap \uhp$ the function has an expansion of the form
\be\label{hol-exp} F(z) \= (z-\bar \z)^{-p}\, \sum_{\nu\geq \mu} a_\nu
\, w^\nu\,, \ee
with $w=\frac{z-\z}{z-\bar\z}$. If $F$ has a singularity at~$\z$ then
$\mu<0$. If $\z\in \Gm i$ then $a_\nu=0$ if
$\nu \not\equiv \frac{p-r}2\bmod 2$, and for $\z\in \Gm e^{\pi i /3}$
there is a similar condition modulo~$3$. The expansion at $\z$ will
represent $F$ only on some open hyperbolic disk around~$\z$ that does
not contain other singularities.

An antiholomorphic modular form $F\in \bar M^S_{-p}(r)$ has similar
expansions:
\badl{ahol-exp} &\text{near }\infty:& F(z) &\= \sum_{\nu\geq \mu}
a_\nu \bar q^{\nu+r/12}\,,\\
&\text{near }\z:& F(z) &\= (\bar z-\z)^{p} \sum_{\nu\geq \mu} a_\nu\,
\bar w^{\nu}\,. \eadl

A $p$-harmonic modular form $F\in H^S_p(r)$ has also expansions at
points of~$\Pt$, the terms of which inherit the $p$-parabolicity:
\begin{align*}
&\text{near }\infty: & F(z) &\= \sum_{\nu \in \ZZ} f_{\infty,\nu}(y)
\, e^{2\pi i(\nu+r/12)x}\,,\\
&\text{near }\z: & F(z) &\= (z-\bar \z)^{-p} \sum_{\nu \in \ZZ}
f_{\z,\nu}(|w|) \, \Bigl( \frac w{|w|}\Bigr)^\nu\,.
\end{align*}
The harmonicity induces second order differential equations for the
coefficients, which then are elements of a two-dimensional space,
with a one-dimensional subspace corresponding to holomorphic terms.
We note that the operator $\x_p$ sends the term with $(w/|w|)^\nu$ to
the term with $(w/|w|)^{\nu+1} = \bar w^{-\nu-1}\, |w|^{\nu+1}$ in
the expansion~\eqref{ahol-exp}, with $p$ replaced by $p-2$.

\section{Real-analytic modular forms}\label{sect-ramf}The task to find
a $p$-harmonic modular form $H$ such that $\x_p H=F$ for a given
antiholomorphic modular form becomes easier if we embed $F$ in a
family of modular forms of a more general type. For this purpose one
may use Poincar\'e series. Here we modify that approach in such a way
that it works for complex weights.

We recall the definition of Maass forms, which are real-analytic
modular forms that satisfy more general conditions than just
(anti)holomorphy or harmonicity. The surjectivity of
$\x_p:H^S_p(r)\rightarrow \bar M^S_{2-p}(r)$ can be reformulated in
terms of Maass forms. To this reformulated problem we will apply
results in~\cite{Br94} that lead to meromorphic families of lifts.

\subsection{Maass forms} We define a third action of
$\Gm/\{\pm I\}=\PSL_2(\ZZ)$ on the functions on $\uhp$:
\begin{defn}\label{Mfdef}For $p,r\in \CC$, $p\equiv r\bmod 2$ and
$\gm = \matc abcd\in \Gm$:
\[ \bigl(f|^\ra_{v_r,p}\gm\bigr)(z) \= v_r(\gm)^{-1}\,
e^{-ip\,\arg(cz+d)} \, f(\gm z)\,. \]
As before, $\arg(cz+d) \in (-\pi,\pi]$.
\end{defn}
This action is intermediate between the actions $|_{v_r,p}$ and
$|_{v_r,p}^a$, and does not favor either holomorphy or
antiholomorphy. Intertwining operators between these actions are
\badl{intertw}
(R_p^h F)(z) &\= y^{p/2} \, F(z)\,,&
\qquad (R_p^h F)|^\ra_{v_r,p}\gm &\= R_p^h\bigl(F|_{v_r,p}\gm\bigr)
&\quad&(\gm\in \Gm)\,,\\
(R_p^a F)(z) &\= y^{-p/2}\, F(z)\,,&
(R_p^a F)|^\ra_{v_r,p}\gm &\= R_p^a\bigl( F|_{v_r,-p}^a
\gm\bigr)&&(\gm\in \Gm)\,. \eadl
The action $|^\ra_{v_r,p}$ of $\Gm$ commutes with the \emph{Casimir
operator in weight~$p$}:
\be \om_p \= -y^2\partial_y^2-y^2\partial_x^2+i p y \partial_x\,.
\ee

We define Maass forms with singularities in a fixed set $S$, which is
a union of finitely many $\Gm$-orbits in~$\uhp$. We choose a system
of representatives $\PtY$ of $\Gm\backslash S$.
\begin{defn}\label{defMf}Let $p,r\in \CC$ with $p\equiv r\bmod 2$. A
\emph{modular Maass form} of weight~$p$ for the multiplier system
 $v_r$ with \emph{spectral parameter} $s$ is a function on
 $\uhp\setminus S$, such that
\begin{enumerate}
\item $f|_{v_r,p}^\ra \gm = f$ for all $\gm\in \Gm$,
\item $\om_p f = \bigl( \frac14-s^2\bigr)\, f$.
\end{enumerate}
By $\mf^S_p(r,s)=\mf^S_p(r,-s)$ we denote the space of such Maass
forms.
\end{defn}
This is a very large space. The definition does not impose growth
conditions. The definition is invariant under $s\mapsto -s$, and we
could work with the \emph{eigenvalue} $\frac14-s^2$. However, in
practice the spectral parameter $s$ is more convenient. As
parametrizations both $s\mapsto \frac14-s^2$ and
$s\mapsto s(1-\nobreak s)$ are in use. Here I choose
$s\mapsto \frac14-s^2$ for easy reference to~\cite{Br94}. In
\cite{Br94} Maass forms are considered as functions on the universal
covering group of $\SL_2(\RR)$. Here we stay on the upper half-plane,
and mention only that to $f\in \mf^S_p(r,s)$ corresponds the function
$p(z)k(\th) \mapsto f(z)\, e^{ip\th}$ in the notations of \S2.2
of~\cite{Br94}.

The operators
\be \label{Epm} \EE^+_p \= 2iy\,\partial_x + 2y\,\partial_y+p\,,\qquad
\EE^-_p \=
-2iy\,\partial_x+2y\,\partial_y-p\,,
\ee
satisfy the relations
\bad \EE^\pm_p \circ \om_p &\= \om_{p\pm 2}\circ \EE^\pm_p\,,\qquad
\EE^\pm_p\bigl( F|^\ra_{v_r,p}\gm \bigr) \= \bigl(\EE_p^\pm
F\bigr)|_{v_r,p\pm 2}^\ra \gm\,,\\
\EE^\pm_{p\mp2}\circ \EE^\mp_p &\= -4\,\om_p-p^2\pm 2p\,, \ead
and give linear maps
\be \EE^\pm_p : \mf^S_p(r,s) \longrightarrow \mf^S_{p\pm 2}(r,s)\,.
\ee

For general combinations of the weight $p$ and the spectral parameter
$s$ these weight shifting operators are bijections between spaces of
Maass forms. Those values of $(p,s)$ where this is not the case are
related to the spaces of modular forms discussed in \S\ref{sect-mf}.
The operators in~\eqref{intertw} lead to the following commuting
diagram:
\be\label{diagS}
\begin{aligned}
\xymatrix{ M_p^S(r)
\ar@{^{(}->}[d] \ar@{^{(}->}[r]^(.45){R^h_p}
& \mf^S_p\bigl(r,\frac{p-1}2\bigr)\ar@{=}[d]
\\
H^S_p(r)
\ar[d]_{\x_p} \ar@{^{(}->}[r]^(.45){R^h_p}
& \mf^S_p\bigl( r,\frac{p-1}2\bigr)
\ar[d]_{-\frac12\EE_p^-}
\\
\bar M^S_{2-p}(r)
\ar@{^{(}->}[r]^(.45){R^a_{p-2}}
&\mf^S_{p-2}\bigl( r,\frac{p-1}2\bigr)
}
\end{aligned}
\ee
The spaces on the right are much larger than those on the left, since
we imposed growth conditions in Definition~\ref{MHSdef} and did not
in Definition~\ref{Mfdef}.

\subsection{Expansions and growth conditions}\label{sect-exp-grc} Any
$f\in \mf^S_p(r,s)$ has a Fourier expansion on a neighborhood
$y>A_\infty$ of~$\infty$ for a suitable $A_\infty>0$, and at each
$\z\in \PtY$ a polar expansion on
$0<\bigl|\frac{z-\z}{z-\bar \z}\bigr|<A_\z$ for suitable $A_\z$. The
 individual terms of these expansions are also eigenfunctions of
$\om_q$ with eigenvalue $\frac14-s^2$. This leads to second order
differential equations, the solutions of which can be described in
special functions. Here we mention the results needed for this paper.
Section~4.2 in~\cite{Br94} gives more information.\medskip

In the Fourier expansion at~$\infty$
\be F(z) \= \sum_{n\equiv r/12\bmod 1} (\Four{\infty,n}f)(z)\,,\qquad
(\Four{\infty,n}f)(z) \= e^{2\pi i n x}\, (\four{\infty,n}f)(y)\,, \ee
the Fourier coefficients $\four {\infty,n}f$ satisfy a second order
differential equation, defining a two-dimensional space of solutions.
If $\re n\neq 0$ this space has a one-dimensional subspace of
elements with quick decay as $y\rightarrow\infty$. As a basis vector
of this subspace we use
\be \om_p(\infty;n,s;z) \= e^{2\pi i n x}\, W_{p\sign(\re n)/2,s}(4\pi
n \sign(\re n) y)\,. \ee
It satisfies
\badl{ominfs} \om_p\bigl(\infty;n,\pm\tfrac{p-1}2;z\bigr)
&\= (4\pi n)^{p/2}\, y^{p/2}\, q^n&&\text{if }\re n>0\,,\\
\om_p\bigl(\infty;n;\pm \tfrac{p+1}2;z\bigr)&\=
(-4\pi n)^{-p/2} y^{-p/2}\bar q^{-n}
&&\text{if }\re n<0\,. \eadl

The other elements in the space are asymptotic to a multiple of
$z\mapsto e^{2\pi i n
x}\, \allowbreak y^{-p\,\sign(\re n)/2}\, \allowbreak e^{2\pi n \sign(\re n)y}$
as $y\rightarrow \infty$. So these terms have exponential growth, of
larger order if $\re n$ gets larger. If $\re n=0$ all element of the
solution space have less than exponential growth.\smallskip

Near $\z\in \PtY$ we have an analogous situation. We have an expansion
\badl{zw} F(z) &\= \sum_{\nu\in \ZZ} (\Four{\z,\nu}f)(z)\,,\\
(\Four{\z,\nu}f)(z) &\= e^{ip\arg(1-w)}\, e^{i\nu\arg w}\,
(\four{\z,\nu}f)(u)\,, \eadl
with $w=(z-\z)/(z-\bar\z)$ and $u=\frac{|w|^2}{1-|w|^2} =
\frac{|z-\z|^2}{4y\, \im \z}$. {\footnotesize

To see that this is the right type of expansion, we compare it with
\eqref{hol-exp} and~\eqref{ahol-exp}, and use the operators
in~\eqref{intertw} to obtain:
\begin{align*}
R^h_p\Bigl( (z-\bar\z)^{-p}\, (w/|w|)^\nu \Bigr)&\= e^{ip\arg(1-w)}\,
e^{i\nu\arg w}\, e^{-\pi i p/2}\, 2^{-p}\, (\im\z)^{-p/2} \,
\bigl(1-|w|^2\bigr)^{p/2} \,,\\
R^a_p\Bigl( (\bar z-\z)^p\, (\bar w/|w|)^\nu \Bigr)&\=
e^{ip\arg(1-w)}\, e^{-i\nu\arg w}\, e^{-\pi i p/2}\,2^p\,
(\im\z)^{p/2}\, \bigl(1-|w|^2\bigr)^{-p/2}\,.
\end{align*}
Thus, we obtain Fourier terms of the form
$e^{ip \arg(1-w)\pm i \nu \arg w}$ times a function on~$u$.}

The Fourier coefficients $\four{\z,n}$ are elements of a
two-di\-men\-sional space, with a one-di\-men\-sional subspace
corresponding to functions without a singularity at~$\z$. This
subspace is spanned by
\bad\label{omz} \om_p(\z;p+2\nu,s;z) &\= e^{ip\arg(1-w)}\, e^{i\nu\arg
w}\, \Bigl( \frac u{u+1}\Bigr)^{\e\nu/2}\, (u+1)^{-s-1/2}\\
&\qquad\cdot
\hypg21\Bigl( \tfrac12+s+\tfrac{\e p}2+\e\nu,\tfrac12+s-\tfrac{\e
p}2;1+ \e\nu;\tfrac u{u+1}\Bigr)\,, \ead
with $\e\in \{1,-1\}$ chosen such that $\e\nu\geq 0$.
(I use the notations of \S4.2 in~\cite{Br94}. A confusing point is
that in \cite{Br94} it made good sense to parametrize the order of
the Fourier terms at $\z\in \PtY$ by $p+2\nu$, while here
parametrization by $\nu$ itself is more convenient.) We have:
\badl{omzts} \om_p\bigl(\z;&p+2\nu,\pm \tfrac{p-1}2;z\bigr)\\
 &\= 2^p \, e^{\pi i p/2}\,
(\im\z)^{p/2}\, y^{p/2}\,
(z-\bar\z)^{-p}\, w^\nu&& \text{if }\nu\geq 0\,,\\
\om_p\bigl( \z;&p+2\nu,\pm \tfrac{p+1}2;z\bigr)\\
&\= 2^{-p}\, e^{\pi i p/2}\, (\im \z)^{-p/2}\,y^{-p/2}(\bar
z-\z)^{p}\, \bar w^{-\nu}&& \text{if }\nu\leq 0\,. \eadl

The other elements that can occur in the term of order $n$ in the
expansion have a singularity at $\z$. This singularity is logarithmic
if $n=p$ and behaves near $w=0$ like $w^{(p-n)/2}$ if $n-p>0$ and
like $\bar w^{(n-p)/2}$ if $n-p<0$.

These results show that the growth of Maass forms can be controlled by
the Fourier expansion. Suppose that $f\in \mf^S_p(r,s)$ satisfies
$f(z) = \oh(e^{ay})$ as $y\rightarrow\infty$ for a given $a>0$. The
Fourier terms $\Four{\infty,n}f$ can be given by a Fourier integral,
and hence satisfy the same estimate. So if $|\re n |>a/2\pi$ then
$\Four{\infty,n}$ has to be a multiple of $\om_p(\infty;n,s)$.
Similarly, if $F$ satisfies near $\z$ the estimate in
\eqref{growth-int} for a certain $a>0$, then all but finitely many
terms in the expansion at $\z$ are multiples of $\om_p(\z;n,s)$.
Conversely, the contribution of the terms in the expansion at
$\x\in \Pt=\{\infty\}\cup\PtY$ that are a multiple of $\om_p(\x;n,s)$
cannot give a large growth.

\begin{defn}
A \emph{growth condition} $\cc$ for $\mf^S_p(r,s)$ is a finite set of
pairs $(\x,n)$ with $\x\in \Pt$ and $n\equiv \frac r{12}\bmod 1$ if
$\x=\infty$, and $n\in \ZZ$ if $\x\in\PtY$. If $\re r \in 12\ZZ$ we
require that $\cc$ contains $(\infty,it)$ for
$t=-i\,\frac{r-\re r}{12}$.\smallskip

Notation: $\cc(\x) = \{ n\;:\; (\x,n)\in \cc\}$.
\end{defn}
A growth condition singles out finitely many terms from the expansions
of Maass forms at points of~$\Pt$. The additional condition can be
understood from the fact that for $\re n=0$ there are no quickly
decreasing non-zero Fourier terms at~$\infty$.

\begin{defn}Let $\cc$ be a growth condition. By $\mf_p^\cc(r,s)$ we
denote the space of $f\in \mf^S_p(r,s)$ that satisfy
\begin{align*}
\Four{\infty,n}f&\;\in\; \CC\, \om_p(\infty;n,s)&&\text{if }n\equiv
\frac r{12}\bmod 1\text{ and }n\not\in \cc(\infty)\,,\\
\Four{\z,\nu}f&\;\in\; \CC\, \om_p(\z;p+2\nu,s)&&\text{if }\nu\in
\ZZ,\; \nu\not\in \cc(\z)\text{ for }\z\in \PtY\,.
\end{align*}
\end{defn}

The weight shifting operators $\EE_p^\pm$ in~\eqref{Epm} behave nicely
with respect to the Fourier expansion at~$\infty$:
\be \EE^\pm_p \Four{\infty,n}f \= \Four{\infty,n}\EE^\pm_p f\,. \ee
For $\z\in \PtY$ we have the more complicated relation
\be \EE^\pm_p \Four{\z,\nu} f \= \Four{\z,\nu\mp1} \EE_p^\pm f\,. \ee
We define for a given growth condition $\cc$ the growth conditions
$\cc^+$ and $\cc^-$ by
\be \cc^\pm(\infty) \= \cc(\infty)\quad\text{ and }\quad \cc^\pm(\z)
\= \bigl\{\nu \mp1\;:\; \nu\in \cc(\z)\bigr\}\quad\text{ if }\z\in
\PtY\,. \ee
The differentiation relations in Table~4.1 on p.~63 of~\cite{Br94}
imply that $\EE_p^\pm$ sends $\om_p(\x;\ast,s)$ to
$\om_{p\pm 2}(\x;\ast,s)$ for all $\x\in \PtY$. Hence
\be\label{EEcc} \EE^\pm_p : \mf_p^\cc(r,s) \rightarrow \mf_{p\pm
2}^{\cc^\pm}(r,s)\,. \ee
See Proposition~4.5.3 in~\cite{Br94}.
(The change in the growth condition is absent in~\cite{Br94}. This is
a consequence of the difference in the parametrization of the order
of terms in the expansions at points of~$\PtY$.)
\smallskip

If $F\in \mf^S_p(r,s)$ satisfies \eqref{expgr} at~$\infty$ and
\eqref{growth-int} at the points in~$\PtY$, then it is in
$\mf_p^\cc(r,s)$ for some growth condition~$\cc$, and conversely each
element of a given $\mf_p^\cc(r,s)$ satisfies those growth conditions
at the points of~$\Pt$. Thus, the diagram~\eqref{diagS} can be
replaced:
\be\label{diagS1}
\begin{aligned}
\xymatrix{ M_p^S(r)
\ar@{^{(}->}[d] \ar@{^{(}->}[r]^(.45){R^h_p}
& \bigcup_\cc \mf_p^\cc\bigl(r,\frac{p-1}2\bigr)\ar@{=}[d]
\\
H^S_p(r)
\ar[d]_{\x_p} \ar[r]^(.45){R^h_p}_(.45){\cong}
& \bigcup_\cc \mf_p^\cc\bigl( r,\frac{p-1}2\bigr)
\ar[d]_{-\frac12\EE_p^-}
\\
\bar M^S_{2-p}(r)
\ar@{^{(}->}[r]^(.45){R^a_{p-2}}
& \bigcup_\cc\mf_{p-2}^{\cc^-}\bigl( r,\frac{p-1}2\bigr)
}
\end{aligned}
\ee
Here $\cc$ runs over all growth conditions for $\mf^S_p(r,s)$.
Moreover,
\bad R^h_p M^S_p(r) &\= \ker\,\Bigl(
\EE^-_p:\bigcup_\cc\mf^\cc_p\bigl(r,\tfrac{p-1}2\bigr) \rightarrow
\bigcup_\cc\mf^{\cc^-}_{p-2}\bigl( r,\tfrac{p-1}2\bigr) \Bigr)\,,\\
R^a_{p-2}\bar M^S_{2-p}(r) &\=\ker\,\Bigl( \EE^+_{p-2}:
\bigcup_\cc\mf^{\cc}_{p-2}\bigl(r,\tfrac{p-1}2\bigr)\rightarrow
\bigcup_\cc\mf_p^{\cc^+}\bigl(r,\tfrac{p-1}2\bigr)\Bigr)\,. \ead
Relation~\eqref{EEcc} shows that the differential operators $\EE_p^+$
and $\EE_p^-$ transform Maass forms satisfying a given growth
condition into Maass forms satisfying a slightly changed growth
condition. So indeed $\x_p H^S_p(r) \subset \bar M^S_{2-p}(r)$.

In the next subsections we will not work with individual Maass forms,
but with families of Maass forms $(r,s) \mapsto f(r,s)$ for $(r,s)$
in some domain $\Om \subset \CC^2$, such that
$f(r,s) \in \mf^S_{\ell+r}(r,s)$ for a given $\ell\in 2\ZZ$. Then we
will use growth conditions $\cc=\{(\x,\nu_0)\}$ 
in which the integers $\nu_0$ determine
functions of $r$. All $\cc(\x)$ are finite subsets of $\ZZ$:
\badl{grcond-fam}\nu_0\in \cc(\infty)\,,\;&\text{ corresponds to }
 r\mapsto \nu_0+\frac r{12}\,,\\
\nu_0 \in \cc(\z)\,,\; \z\in \PtY\,,\;&\text{ corresponds to }
r\mapsto \nu_0\,. \eadl
The variable $r$ should run over a set $U\subset \CC$ such that
$\nu+\frac{\re r}{12}\neq 0$ for all $\nu \in \cc(\infty)$
and~$r\in U$. In this context we interprete $\four{\x,\nu}$ as
$\four{\z,\nu+r/12}$ if $\x=\infty$ and as $\four{\x,\nu}$ if
$\x\in \PtY$.

\subsection{Perturbation theory}The basis for our proof of
Theorems~\ref{thmA}--\ref{thmC} is Theorem 9.4.1 in~\cite{Br94}. It
gives meromorphic families of Maass forms with a prescribed behavior
of the terms in the expansions at points of~$\Pt$ given by a growth
condition. In~\cite{Br94} it is a step in obtaining the meromorphic
continuation of Poincar\'e series in $(r,s)$ jointly. In this paper
it is convenient to use this intermediate result and not the
continued Poincar\'e series.\medskip

For our given sets $S$ and $\Pt$ we have expansions of Maass forms on
regions $y>A_\infty$ near~$\infty$ and on regions
$0<\bigl|\frac{z-\z}{z-\bar\z}\bigr|<A_\z$ near $\z\in \PtY$. These
regions may overlap. To be able to apply the results in Chapters~7--9
of~\cite{Br94} we shrink these regions such that their images in the
quotient $\Gm\backslash\uhp$ are pairwise disjoint.

We choose for each $\x\in \Pt$ a \emph{truncation point} $a_\x$ such
that $a_\infty >A_\infty$ and
$a_\z\in \bigl(0,A_\z^2/(1-\nobreak A_\z^2) \bigr)$ for $\z\in \PtY$.
Hence $(\four{\infty,n}f)(a_\infty)$ and $(\four{\z,\nu}f)(a_\z)$ are
well defined. The precise choice of the truncation points does not
matter.

We use a sequence $(\cc_N)_{N\geq 1}$ of growth conditions as
in~\eqref{grcond-fam}, with
\be\label{grc-inf} \cc_N(\infty)\= \{\nu\in \ZZ\;:\; |\nu|<N\}\,,\ee
and with finite sets $\cc(\z)$ for $\z\in \PtY$ that do not depend
on~$N$. We formulate part of the statement of Theorem~9.4.1
in~\cite{Br94}:
\begin{thm}\label{thm-fam-e} There is an open disk
$V_{0,N}=V_0(\cc_N)$ around $0$ in $\CC$ such that for each
collection of holomorphic functions
$\rho=\bigl(\rho_{\x,\nu}\bigr)_{(\x,\nu)\in \cc_N}$ on
$V_{0,N}\times\CC$ there is a unique meromorphic family $e_\rho$ of
Maass forms on $V_{0,N}\times\CC$ with values in $\mf^{\cc_N}_\ell$
satisfying
\be\label{Fe-a} \four{\x,\nu} e_\rho(r,s;a_\x) \= \rho_{\x,\nu}(r,s)
\qquad\text{ for all }(\x,n)\in \cc_N\,.\ee
\end{thm}

A family $(r,s)\mapsto f(r,s)$ of Maass forms is holomorphic if it is
pointwise holomorphic and also all terms $\four{\x,n}f(r,s)$ in the
expansions at $\infty$ and~$\z$ are pointwise holomorphic. It has
values in $\mf^\cc_\ell$ if its value at $(r,s)$ is in
$\mf^\cc_{\ell+r}(r,s)$ for each $(r,s)$ in its domain.

A meromorphic family $f$ on $V_{0,N}\times\CC$ with values in
$\mf_\ell^{\cc_N}$ is not just a family that is pointwise meromorphic
on $\uhp\setminus S$. We require that locally on its domain the
family can be written as $\frac 1\ps \, h$, where $h$ is a
holomorphic family with values in $\mf_\ell^{\cc_N}$ and $\ps$ is a
non-zero holomorphic function. The idea is that ``denominators should
not depend on~$z$''.

We call $(r_0,s_0)$ a \emph{singularity} of the family if the family
is not a holomorphic family on a neighborhood of $(r_0,s_0)$. So
$\ps(r_0,s_0)$ should vanish for the representation
$f=\frac 1\ps\, f$ that is valid on a neighborhood of $(r_0,s_0)$.

It should be noted that although the functions $\rho_{\x,n}$ are
holomorphic, the families $(r,s)\mapsto \four{\x,\nu}e_\rho(r,s)$ are
meromorphic. Their singularities are not visible in the functions
$(r,s)\mapsto \four{\x,\nu}e_\rho(r,s;a_\x)$.

The theorem is based on the existence of a holomorphic family
$r\mapsto A^a(r)$ on a disk in~$\CC$ centered at~$0$ of self-adjoint
operators in a Hilbert space. The Hilbert space and the family depend
on the growth condition $\cc_N$ and the truncation points $a_\x$ for
$\x\in \Pt$. After preparations in earlier chapters it is defined (in
a more general context)
in \S9.2 of~\cite{Br94}. It is a generalization of the pseudo Laplace
operator of Colin de Verdi\`ere in~\cite{CdV}. The family $A^a$ can
be studied with the methods of analytic perturbation theory in Kato's
book~\cite{Ka}. Eigenvectors of $A^a(r)$ with eigenvalue
$\frac14-s^2$ correspond to Maass forms
$F\in \mf_{\ell+r}^{\cc_N}(r,s)$ for which $(\four{\x,n} f)(a_\x)=0$
for all $(\x,n)\in \cc_N(r)$.

The resolvent gives a meromorphic family $(r,s)\mapsto R^a(r,s)$ of
bounded operators. This resolvent is used in the construction of
$e_\rho$ in the theorem. We do not know much about this resolvent,
except that it is meromorphic, and we have some eigenvalue estimates
that give information on its singularities for $r\in \RR\cap V_0$.
This gives the following additional information:
\begin{lem}\label{lem-erho-ap}Let $e_\rho$ be as in
Theorem~\ref{thm-fam-e}. For each $r\in V_{0,N}$ the set of
$s\in \CC$ such that $e_\rho$ has a singularity at $(r,s)$ is
discrete in~$\CC$. If $(r,s)$ is a singularity of $e_\rho$ with
$r\in V_0\cap\RR$ then $\frac14-s^2
\geq -\frac14\,\bigl(\ell+\nobreak r)^2$.
\end{lem}
\begin{proof}Theorem~9.4.1 in~\cite{Br94} states that each singularity
of $e_\rho$ is a singularity of the resolvent $R^a$. This gives the
first assertion. The eigenvalue estimate follows from 9.2.1
in~\cite{Br94}.
\end{proof}

\begin{lem}\label{lem-V0N}The $V_{0,N}:=V_0(\cc_N)$ can be chosen to
form an increasing collection of open neighborhoods of $0$ in~$\CC$
satisfying $V_{0,N} \subset \{r\in \CC\;:\; |\re r|<12N\}$ and
\be \bigcup_{N\geq 0} V_{0,N}\= \CC\,. \ee
\end{lem}
{\footnotesize\begin{proof}If $r\in V_{0,N}$ would not satisfy
$|\re r|<12(N+\nobreak 1)$ then $\cc_N$ would not be a suitable
growth condition. To see that the union of the sets $V_{0,N}$
equals~$\CC$ we have to go into some details of the reasoning
in~\cite{Br94}.

We start at the proof of Lemma~9.1.6 in~\cite{Br94}. There it is
indicated that the set $V_0$ should consist of $r\in \CC$ such that
\be\label{V0-cond} b_{1,b}\, d_b \, |r| + b_{1,c}\, d_c \, |r| +
b_{2,c} \, d_c^2\, |r|^2 \;<\; 1\,, \ee
with positive factors that we have to trace back through the lemmas
in~\cite{Br94}.

The linear form $\ph=\ph_r$ in Lemma~9.1.6 is of the form $r \al$,
with $\al$ as explained in~13.4.7. We define the factors $d_b$ and
$d_c$ as $\|\al\|_b=\|\ph_r\|_b/|r|$ and $\|\al\|_c=\|\ph\|_c/|r|$.
From 8.4.10 we see that $d_c>0$ depends only on the group. For $d_b$
we go from Lemma~8.4.11 via the definition of $b_\ph^\pm$ in 8.2.3 to
the function $t_\ph$ in Lemma~8.2.1. There we see that $t_\ph$
depends on the set $\Pt$, but not on the actual growth condition.
Hence $d_b>0$ is also $\oh(1)$, independent of~$N$ and the finite
sets $\cc_N(\z)\subset\ZZ$ for $\z\in \PtY$.

The constants $b_{\ast,\ast}$ are given in Lemma~9.1.5, and expressed
in a large quantity $\x$, which depends on $N$, on an arbitrary small
quantity~$\e$, and on a positive quantity~$n_1$. In the proof of
Lemma~8.4.11 the quantity $n_1$ is defined depending on the group
only. This means that Lemma~9.1.5 gives
\[ b_{1,b}\= \oh(\e^2)\,,\quad b_{1,c}\=
\oh(\x^{-1})+\oh(\e)+\oh(\e^{7/4}\x^{-1/4})\,,\quad b_{2,c} \=
\oh(\x^{-2})\,.\]
The dependence on~$N$ is via~$\x$, and possibly via our choice
of~$\e$.

The definition of $\x$ in Lemma~8.4.11 gives for the present situation
\[ \x \= 2\pi N - \frac{\ell}{2a_\infty} \;\geq \; 2\pi N\,.\]
(We use that $\ell\leq 0$.)

Taking $\e=\frac 1N$ we see that there is $C>0$, not depending on~$N$,
such that
\[ \frac{|r|}N + \Bigl( \frac{|r|}N\Bigr)^2<C\]
implies that \eqref{V0-cond} is satisfied. Determining $V_{0,N}$ by
\be \label{V0}|r|< N\, \min\Bigl(12, \, \frac{\sqrt{1+4C}-1}2\Bigr)
\ee
we satisfy all conditions.
\end{proof}\par}

\subsection{Family of antiholomorphic modular forms with
singularities} Any element of $\bar M^\mer_{2-p_1}(r_1)$ with
$p_1\equiv r_1\bmod 2$ can be written as $\bar \eta^{-2r} F$ with
$r\in \CC$ and $F\in \bar M^\mer_{2-\ell}(0)$ with
$\ell\in 2\ZZ_{\leq 0}$; so $F$ is the conjugate of a meromorphic
automorphic form of even weight $2-\ell$ for the trivial multiplier
system. We consider the holomorphic family
$r\mapsto \bar\eta^{-2r}F$. For each $r\in \CC$ we have
$\bar\eta^{-2r}F\in \bar M^\mer_{2-\ell-r}(r)$. There is freedom in
the choice of $\ell$ and~$r$; we use it to take $\ell\leq 0$. In the
remainder of this section we construct a $(\ell+\nobreak r)$-harmonic
lift of $\bar \eta^{-2r}\, F$ depending meromorphically on~$r$.

Let $S$ be the $\Gm$-invariant set of points in~$\uhp$ at which $F$
has a singularity. Then also $\bar\eta^{-2r} F$ has its singularities
in~$S$. We form $\Pt=\{\infty\} \cup\PtY$ where $\PtY$ is a system of
representatives of $\Gm\backslash S$.

In \eqref{ahol-exp} we have seen that the function $\bar \eta^{-2r}F$
has expansions of the following form
\badl{frexp} \text{at }\infty:& & (\bar \eta^{-2r}F)(z) &\=
\sum_{\nu\geq \mu_\infty} a_\nu^\infty(r)\, \bar q^{\nu-r/12}\,,\\
\text{at }\z\in \PtY:&&
(\bar \eta^{-2r}F)(z) &\=(\bar z-\z)^{\ell+r-2} \sum_{\nu \geq \mu_\z}
a^\z_\nu(r)\, \bar w^{\nu}\,, \eadl
with $\bar q=e^{-2\pi i \bar z}$ and
$\bar w=\frac{\bar z-\bar\z}{\bar z-\z}$. The coefficients $a_\nu^\x$
are holomorphic functions on~$\CC$.
(The $a_\nu^\infty$ are actually polynomials in $r$ of degree at most
$\nu-\mu_\infty$.) Since we have chosen $S$ as the set of
singularities of~$F$ in~$\uhp$, we have $\mu_\z\leq -1$ for all
$\z\in \PtY$.

The family
\be f_r \= R^a_{2-\ell-r}\bigl( \bar\eta^{-2r}F\bigr)\,,\qquad f_r(z)
\= y^{1-(\ell+r)/2}\, {\overline{\eta(z)}\,}^{-2r}\, F(z)
\ee
is a holomorphic family on~$\CC$ of Maass forms, with $f_r\in
\mf^S_{\ell+r-2}\bigl(r,\frac{\ell+r-1}2\bigr)$. A comparison with
\eqref{ominfs} and \eqref{omzts} shows that
\begin{align*} (\Four{\infty,-\nu+r/12}&f_r)(z) \= a_\nu^\infty(r)\,
y^{1-(\ell+r)/2}\, \bar q^{\nu-r/12}
\qquad\text{for all }\nu\geq \mu_\infty\,,\\
&\= a_{\nu}^\infty(r)\, \Bigl( 4\pi\bigl( \nu-\tfrac r{12}\bigr)
\Bigr)^{(\ell+r)/2-1}\, \om_{\ell+r-2}\bigl(\infty;-\nu+\tfrac
r{12},\tfrac{\ell+r-1}2\bigr)\\
&\qquad\text{if }-\nu+\frac{\re r}{12}<0\,,
\displaybreak[0]
\\
(\Four{\z,-\nu}&f_r) (z) \= a_\nu^\z(r)\, y^{1-(\ell+r)/2}\, (\bar
z-\z)^{\ell+r-2}\, \bar w^\nu\qquad \text{for all }\nu\geq \mu_\z\,,
\\
&\= -a_{\nu}^\z(r)\, e^{-\pi i(\ell+r)/2}\,
(4\im \z)^{(\ell+r)/2-1}
\om_{\ell+r-2}\bigl(\z;-\nu;\tfrac{\ell+r-1}2;z) \\
&\qquad\text{for }\z\in \PtY\,, \text{ if }\nu \geq 0\,.
\end{align*}
We take the growth condition $\cc_N$ as indicated in \eqref{grc-inf}
and~\eqref{grcond-fam} with
\bad \cc_N(\z) &\= \bigl\{\nu\in \ZZ\;:\;1\leq \nu \leq
-\mu_\z\}\quad\text{ for }\z\in \PtY,\\
&\;\;\text{ and }\quad N\geq \max(1,1-\mu_\infty)\,. \ead
Then $f_r \in \mf_{\ell+r-2}^{\cc_N}\bigl( r,\frac{\ell+r-1}2\bigr)$
for all $r\in V_{0,N}$, for the disk $V_{0,N}$ in
Theorem~\ref{thm-fam-e}.

For $(\x,\nu)\in \cc_N(r)$ the families
$r\mapsto (\four{\x,\nu} f_r)(a_\x)$ are holomorphic multiples of
$a_{-\nu}^\x(r)$. This means that the functions
\be \label{rho-choice}
 \rho_{\x,\nu}(r,s) \= (\four{\x,\nu} f_r)(a_\x)\ee
are holomorphic on $\CC$ for all $(\x,\nu)\in \cc_N$. We apply
Theorem~\ref{thm-fam-e} and obtain a meromorphic family $e_\rho$ of
Maass forms of $V_{0,N}\times \CC$ with Fourier coefficients
determined by $\cc_N$ satisfying~\eqref{Fe-a}.

Now we ask whether the restriction of $e_\rho$ to the complex line
$s=\frac{\ell+r-1}2$ has anything to do with the family $f_r$, apart
from relation \eqref{Fe-a}. The first worry is that $e_\rho$ might
have a singularity along $s=\frac{\ell+r-1}2$, which would mean that
there is no meromorphic restriction to this line at all.

Suppose that there were such a singularity carried along the line
$s=\frac{\ell+r-1}2$. Take the minimal integer $k\geq 1$ such that
\[ p(r) \= \lim_{s\rightarrow(\ell+r-1)/2} \bigl( s -
\tfrac{\ell+r-1}2\bigr)^k \, e_\rho(r,s)\]
exists for a dense set of $r\in V_{0,N}$. Since $k$ is minimal, $p(r)$
is non-zero for some $r$, and hence the meromorphic family
$r\mapsto p(r)$ of Maass forms is non-zero. For each
$(\x,\nu)\in \cc_N$ we have $\bigl(\four{\x,\nu} p(r)\bigr)(a_\x)=0$.
So $p$ is a meromorphic family of eigenfunctions of the family of
operators $A^a(\cdot)$. Lemma~\ref{lem-erho-ap} implies that for
those $r\in V_{0,N}\cap\RR$ at which it has no singularity we have
\[ \frac14 - \Bigl( \frac{\ell+r-1}2\Bigr)^2 \;\geq \; -\frac14
(\ell+r)^2\,.\]
This cannot be true for $r\in V_{0,N}\cap(-\infty,0)$, since we have
taken $\ell\leq 0$. Hence $k=0$ and $e_\rho$ has a restriction to the
line $s=\frac{\ell+r-1}2$. This restriction may be meromorphic; the
good thing is that it exists at all.

Next, we check that the restriction is equal to $f_r$. We consider the
meromorphic family
$p_1:r\mapsto e_\rho\bigl( r,\frac{\ell+r-1}2\bigr) - f_r$. It might
be a non-zero family. We know from Theorem~\ref{thm-fam-e} that
$r\mapsto \bigl(\four{\x,\nu} p_1(r)\bigr)(a_\x)$ is the zero
function for all $(\x,\nu)\in \cc_N$. Again by the eigenvalue
estimate in Lemma~\ref{lem-erho-ap} this is impossible. Hence $f_r$
is equal to the restriction of the family $e_\rho$ to the line
$s=\frac{\ell+r-1}2$.

\subsection{Lift of the family}The advantage of describing
$r\mapsto f_r$ as the restriction of a family of Maass forms in two
variables, is that it is easier to lift such a family.

In diagram~\eqref{diagS} we see that we want to find $h_r$ such that
$-\frac12 \EE_{\ell+r}^- h_r = f_r$. The differential operator
$\frac14 \EE_{\ell+r}^- \EE_{\ell+r-2}^+$ acts on the space
$\mf_{\ell+r-2}^{\cc_N}(r,s)$ as multiplication by
\[ \Bigl( s-\frac{\ell+r-1}2\Bigr)\,\Bigl(
s+\frac{\ell+r-1}2\Bigr)\,.\]

So let us consider
\be h(r,s) \= - \frac12\, \Bigl( s-\frac{\ell+r-1}2\Bigr)^{-1} \Bigl(
s+\frac{\ell+r-1}2\Bigr)^{-1}\, \EE_{\ell+r-2}^+\, e_\rho(r,s)\,.\ee
This is well defined as a meromorphic family on $V_{0,N}\times \CC$.
Since the family $\bar\eta^{-2r}\, F$ that we started with is
antiholomorphic, we have $\EE_{\ell+r-2}^+\, f_r=0$. Hence
$\EE_{\ell+r-2}^+ e_\rho(r,s)$ has a zero along $s=\frac{\ell+r-1}2$,
which cancels the factor $\left(s-\frac{\ell+r-1}2\right)^{-1}$, and
hence $h_r$ has no singularity carried by the line
$s=\frac{\ell+r-1}2$. So the restriction
$h_r = h\bigl( r,\frac{\ell+r-1}2\bigr)$ exists as a meromorphic
family of Maass forms on $V_{0,N}$ and satisfies
$-\frac12 \EE_{\ell+r}^- h_r = f_r$. It is a family that satisfies
$h_r \in \mf_{\ell+r}^{\cc_N^+}\bigl(r,\frac{\ell+r-1}2\bigr)$.

Going back by the operator $R^h_{\ell+r}$ in diagram~\eqref{diagS1} we
get a meromorphic family $H_r = \bigl(R^h_{\ell+r}\bigr)^{-1} h_r$ of
harmonic modular forms. We obtain the following intermediate result:
\begin{prop}\label{prop-mero-hNr}Let $F$ be an antiholomorphic modular
form of weight $2-\ell\in 2\ZZ_{\geq 1}$ with singularities in the
set $S\subset\uhp$. Then there is a collection
$\{ V_{0,N}\;:\; N \geq N_F\}$, for some $N_F\geq 1$, of open
neighborhoods of $0$ in~$\CC$ with the properties in
Lemma~\ref{lem-V0N}, such that on each $V_{0,N}$ there is a
meromorphic family $H_{N,r}$ of harmonic automorphic forms such that
\[ H_{N,r} \in H^S_{\ell+r}(r)\quad\text{ and }\quad \x_{\ell+r}
H_{N,r} = \bar\eta^{-2r}\, F\]
for each $r\in V_{0,N}$ at which $H_{N,r}$ is defined.
\end{prop}
The construction gives $H_{N,r}$ far from uniquely. It depends on the
choice of the truncation parameters $a_\x$, and in~\eqref{rho-choice}
we could have taken other holomorphic functions with the same
restrictions to the line $s=\frac{\ell+r-1}2$.

Let $\Pt=\{\infty\}\cup\PtY$ as before. The family $H_{r,N}$ that we
have constructed satisfies the growth condition~$\cc_N^+$. The growth
condition depends on the order $-\mu_\z$ of the singularities of $F$
at $\z\in \PtY$. (Here the order is determined by the lowest power of
$\bar w$ occurring in the expansion, even if $\z$ is in an elliptic
orbit.) The order of the term $\bar q^{\mu_\infty}$ at which the
expansion of $F$ at~$\infty$ starts determines the lower bound
$N_F=\max(1,1-\nobreak \mu_\infty)$ in the proposition.

We summarize the information that we now have. All terms in the
Fourier expansion
\be\label{Fourexph} H_{N,r}(z) \= \sum_{\nu \in \ZZ}
H_{N,\nu}^\infty(r;z)\,,\qquad H_{N,\nu}^\infty(r;z+x') \= e^{2\pi i
(\nu+r/12)x'}\, H_{N,\nu}^\infty(r;z)\,, \ee
are meromorphic in $r\in V_{0,N}$ and satisfy
\badl{mer-exp-inf} \x_{\ell+r} H_{N,\nu}^\infty(r;z) &\=
a_{-\nu}^\infty(r) \, \bar q^{-\nu-r/12}\,,\\
H_{N,\nu}^\infty(r;z) &\= b_{N,\nu}^\infty(r) \,
q^{\nu+r/12}\qquad\text{ if }\nu \geq N\,,\\
&\= a_{-\nu}^\infty(r) \, \Bigl( -4\pi\bigl( \nu+\tfrac
r{12}\bigr)\Bigr)^{\ell+r-1}\, q^{\nu+r/12}\\
&\qquad\hbox{} \cdot
 \Gf\Bigl(1-\ell-r,-4\pi\bigl( \nu+\tfrac r{12}\bigr) y \Bigr)
\qquad\text{ if }\nu\leq -N\,, \eadl
with the convention that $a_\mu^\infty=0$ if $\mu<\mu_\infty$. The
meromorphic functions $b_{N,\nu}^\infty$ are unknown, the functions
$a_{-\nu}^\infty$ are holomorphic on~$\CC$ and occur in the Fourier
expansion \eqref{frexp} of $\bar\eta^{-2r}F$. The incomplete
gamma-function $\Gf(p,t) = \int_{u=t}^\infty u^{p-1}\, e^{-u}\, du$
is obtained from specialization of the Whittaker function
in~\eqref{ominfs}. Anyhow, the expression for $\nu<-N$ gives the
unique quickly decreasing function with the right image
under~$\x_{\ell+r}$.

Now let $\z\in \PtY$. The expansion at $\z$ has the form
\badl{ztexph} H_{N,r}(z) &\= (z-\bar\z)^{-\ell-r}\sum_{\nu\in \ZZ}
H_{N,\nu}^\z \Bigl(r;\frac{z-\z}{z-\bar\z}\Bigr)\,,\\
 H_{N,\nu}^\z( t w) &\= t^\nu \, H_{N,\nu}^\z(w) \qquad \text{ for
}|t|=1\,. \eadl
The $H_{\nu,r}^\z$ are meromorphic families on $V_{0,N}$ with
$(\ell+\nobreak r)$-harmonic values, and satisfy
\badl{mer-exp-Y} \x_{\ell+r} \biggl( (z-\bar\z)^{-\ell-r}\, &
H_{N,\nu}^\z\Bigl(r;\frac {z-\z}{z-\bar \z}\Bigr) \biggr) \= (\bar
z-\z)^{\ell+r-2}\, a_{-\nu-1}^\z(r) \, \Bigl( \frac{\bar
z-\bar\z}{\bar z-\z}\Bigr)^{-\nu-1}\,,\\
H_{N,\nu}^\z(w) &\= b_{N,\nu}^\z(r) \, w^{\nu}\qquad\text{ if }\nu\geq
-\mu_\z\,,\\
&\= a_{-\nu-1}^\z(r) \, (4\im\z)^{\ell+r-1} \, w^{\nu}\, \Bt\bigl(
|w|^2,-\nu,1-\ell-r\bigr)\\
&\qquad\qquad \qquad \text{ if }\nu\leq -1\,. \eadl
The meromorphic functions $b_{N,\nu}^\z$ are unknown. The coefficients
$a_{-\nu-1}^\z$ from the expansion of $\bar\eta^{-2r}F$ are
holomorphic on~$\CC$. Here we meet the incomplete beta-function
\be \Bt(t,a,b) \= \int_0^t u^{a-1}\, (1-u)^{b-1}\, du \qquad(0\leq t
<1\,,\; \re a>0)\,. \ee
It arises from specialization of the hypergeometric function
in~\eqref{omz}.\medskip

At this point the Maass forms have served their purpose. We have
obtained a meromorphic family of harmonic lifts
of~$\bar\eta^{-2r}\,F$.

\section{Holomorphic families of harmonic forms}\label{sect-hfhf} In
the previous section we have used the analytic perturbation theory of
automorphic forms to construct meromorphic families of harmonic lifts
of families $r\mapsto \bar\eta^{-2r}F$. In this section we modify
these families to obtain holomorphic families of lifts, and extend
them to larger domains than the disks $V_{0,N}$ in
Theorem~\ref{thm-fam-e}. This will bring us to the main result,
Theorem~\ref{thm-main}.

\subsection{Freedom in the choice of the lifts}\label{sect-frd} The
families $r\mapsto H_{N,r}$ may be far from unique. We have the
freedom to add a meromorphic family of holomorphic modular forms.

Adding such a family $r\mapsto A_r$ on $V_{0,N}$ should not change the
description of the Fourier terms of $H_{N,r}$ in \eqref{mer-exp-inf}
and~\eqref{mer-exp-Y}. At $\infty$ the family $r\mapsto A_r$ should
have a Fourier expansion of the form
$\sum_{\nu \geq 1-N} c_\nu(r) \, q^{\nu+r/12}$, and at each
$\z\in \PtY$ an expansion $(z-\nobreak\bar\z)^{-\ell-r}\,\allowbreak 
 \sum_{\nu\geq 0} c_\nu(r) \, w^\nu$.

So $A_r$ should not have singularities at points of $S\subset\uhp$,
and hence should be in $M^!_{\ell+r}(r)$. Each such $A_r$ has the
form $A_r = \eta^{2r-24m_\ell}\,E_k\, p(J)$ where
$\ell= k-12 m_\ell$, $m_\ell\in \ZZ$, $k\in \{0,4,6,8,10,14\}$, with
$E_0=1$ and $E_k$ is the holomorphic Eisenstein series in weight~$k$
otherwise, and where $p(J)$ is a polynomial in the modular
invariant~$J$. The Fourier expansion of $\eta^{2r-24m_\ell} \, E_k$
starts with $q^{r/12-m_\ell}$, so the polynomial $p$ should have
degree at most $N-1-m_\ell$. This determines the dimension $N-m_\ell$
of the space of holomorphic modular forms that we can add.

It may happen that $N-m_\ell\leq0$. Then $r\mapsto H_{N,r}$ is the
unique meromorphic family on $V_{0,N}$ of
$(\ell+\nobreak r)$-harmonic lifts of $r\mapsto \bar\eta^{-2r}F$ with
expansions of the required type.

Let $N-m_\ell\geq 1$. We start with the holomorphic families
$r\mapsto \eta^{2r-24m_\ell}\, E_k \, J^{a-m_\ell}$ on $\CC$ of
holomorphic modular forms, with $a\geq m_\ell$, and form successively
linear combinations $r\mapsto j_{\ell,a,r}$, such that $j_{\ell,a,r}$
has a Fourier expansion of the form
\be\label{jellradef}
j_{\ell,a,r}(z) \= q^{r/12-a} + \sum_{\nu \geq 1-m_\ell}
c_{\ell,a,\nu}(r)\, q^{\nu+r/12}\,. \ee
These families and the coefficients $c_{\ell,a,\nu}$ are holomorphic
on~$\CC$. We have the freedom to add to the family $r\mapsto H_{N,r}$
a meromorphic linear combination of the $j_{\ell,a,r}$ with
$m_\ell\leq a \leq N-\nobreak1$.

\subsection{Removal of singularities}
\begin{prop}For each family $r\mapsto H_{N,r}$ as in
Proposition~\ref{prop-mero-hNr} with expansions as described in
\eqref{mer-exp-inf} and~\eqref{mer-exp-Y} there is a meromorphic
family $r\mapsto A_r$ on $V_{0,N}$ of meromorphic modular forms such
that $r\mapsto H_{N,r}-A_r$ is a holomorphic family of harmonic forms
that has expansions of the form indicated in~\eqref{mer-exp-inf}
and~\eqref{mer-exp-Y}.
\end{prop}
\begin{proof}We describe the Fourier terms $H^\infty_{N,\nu}$ with
$-N<\nu<N$ in terms of explicit special functions. Specializing the
family of functions in 4.2.5 in~\cite{Br94}, we arrive at the
following terms in the expansion at~$\infty$:
\begin{align}
\label{Hei}
&\qquad\text{ for }-N<\nu<N:\\
\nonumber
H^\infty_{N,\nu}(r;z) &\= \frac{a^\infty_{-\nu}(r)}{\ell+r-1}\,
y^{1-\ell-r}\, q^{\nu+r/12}\, \hypg11\Bigl(
1-\ell-r;2-\ell-r;4\pi\bigl(\nu+\tfrac r{12}\bigr)y\Bigr)
\\
\nonumber
&\qquad\qquad\hbox{}
+ b_{N,\nu}^\infty( r) \, q^{\nu+r/12}\,,
\end{align}
with meromorphic functions $b_{N,\nu}^\infty$ on~$V_{0,N}$. The basis
functions that we use may have singularities at points of
$V_{0,N}\cap \ZZ$, and the identities for $H_{N,\nu}^\infty$ are
understood as identities of meromorphic functions of~$r$.

Now we form the meromorphic family
\[ \tilde H_{N,r}(z) \= H_{N,r}(z) - \sum_{a=m_\ell}^{N-1}\,
b_{N,-a}^\infty(r) \, j_{\ell,a,r}(z)\,,\]
with $j_{\ell,a,r}$ as in~\eqref{jellradef}. The new family
$\tilde H_{N,r}$ has the same properties as $H_{N,r}$, with Fourier
terms as in \eqref{Hei}, but now with $b_{N,\nu}^\infty=0$ for
$1-N\leq \nu \leq
-m_\ell$. If $m_\ell\geq N$, then $\tilde H_{N,r}=H_{N,r}$.

Suppose that $\tilde H_{N,r}$ has a singularity at $r_0\in V_{0,N}$.
Then there is an integer $k\geq 1$ such that the meromorphic family
\be\label{rth} q(r) \= (r-r_0)^k \, \tilde H_{N,r}\ee
is holomorphic on a neighborhood of $r_0$ in~$V_{0,N}$, with non-zero
value $q(r_0)$. The $(\ell+\nobreak r_0)$-harmonic modular form
$q(r_0)$ satisfies
$\x_{\ell+r_0} q(r_0) = (r_0-r_0)^k\, \bar\eta^{-2r_0}\, F=0$. So
$q(r_0) \in M_{\ell+r_0}^\mer(r_0)$. In its Fourier expansion there
are non-zero multiples of $q^{\nu+r_0/12}$ with $\nu\leq -m_\ell$. On
the other hand, the Fourier expansion of $\tilde H_{N,r}$ shows that
$q(r_0)$ can have only multiples of $q^{\nu+r_0/12}$ with
$\nu>-m_\ell$, unless the factor
\[ \frac{a_{-\nu}^\infty(r)}{\ell+r-1}\,
\hypg11\Bigl(1-\ell-r;2-\ell-r;4\pi \bigl(\nu+\tfrac
r{12}\bigr)\Bigr)\]
has a singularity at~$r=r_0$. So if $r_0\not \in V_{0,N}\cap\ZZ$, then
$\tilde H_{N,r}$ cannot have a singularity at~$r_0$.

We have arrived at the knowledge that the family
$r\mapsto\tilde H_{N,r}$ has at most finitely many singularities in
$V_{0,N}$, which we can attack one by one. Suppose that $H_{N,r}$ has
a singularity of order $k\geq 1$ at $r_0\in V_{0,N}$. (It does not
matter anymore that then $r_0$ is an integer.)
Define $q$ as in~\eqref{rth}. Then
$q(r_0) \in M^\mer_{\ell+r_0}(r_0)$. Its Fourier expansion has no
terms $q^{\nu+r_0/12}$ with $\nu\leq -N$ and at points $\z\in \PtY$
it has a pole of order at most $-\mu_\z$. We subtract from $\tilde
H_{N,r}$ the family
\[ p_{r_0}:r\mapsto \frac1{(r-r_0)^k}\, \eta^{2(r-r_0)}\, q(r_0)\,.\]
This is a meromorphic family on $\CC$ of meromorphic modular forms.
The family $r\mapsto \tilde H_{N,r}-p_{r_0}(r)$ satisfies the
properties as $r\mapsto H_{N,r}$, and has at $r_0$ a singularity of
order strictly less than~$k$.

Proceeding in this way we remove all remaining singularities of the
family $r\mapsto \tilde H_{N,r}$ in finitely many steps, thus
completing the proof of the proposition.
\end{proof}

\subsection{Normalization}\label{sect-norm} Now we can suppose that
the family $r\mapsto H_{N,r}$ is holomorphic. If $m_\ell\leq N-1$
there is still the freedom of adding holomorphic multiples of the
families $j_{\ell,a,r}$ in~\eqref{jellradef}. We use this freedom to
normalize the Fourier expansion further, in order to compare the
families for different values of~$N$. To do this the confluent
hypergeometric function in~\eqref{Hei} is inconvenient, since it has
singularities as a function of $\ell+r$. Instead we use the following
function:
\bad\label{Mdef} M_p(n;y) &\=\frac {y^{1-p}}{p-1}\, \hypg11\Bigl(
1-p;2-p;n y\Bigr)
+ \sum_{k=0}^\infty \frac{n^k}{k!\;
(1+k-p)}\\
&\= \int_{t=y}^1 t^{-p}\, e^{nt}\, dt\,. \ead
This is holomorphic as a function of $p\in \CC$. The sum on the first
line converges absolutely for all $r\in \CC$, defining a meromorphic
function on~$\CC$ with the opposite principal parts as the term with
the confluent hypergeometric function.

With these slightly more complicated basis functions we write
\eqref{Hei}, with $-N<\nu<N$ as
\be H^\infty_{N,\nu} (r;z) \= \Bigl( a_{-\nu}^\infty(r) \,
M_{\ell+r}\bigl(4\pi\bigl( \nu+\tfrac r{12}\bigr);y\bigr) + \hat
b^\infty_{N,\nu}(r)\Bigr) \, q^{\nu+r/12)}\,, \ee
with the convention that $a_\nu^\infty=0$ for $\nu<\mu_\infty$.
Subtracting suitable multiples of $j_{\ell,-\nu,r}$ with
$1-N\leq \nu\leq -m_\ell$ we arrange that
$\hat b^\infty_{N,\nu}(r) = 0$ for $1-N\leq  \nu \leq -m_\ell$. Thus
we arrive at the following normalization:
\begin{prop}\label{prop-hol-norm}Let $\ell\in2\ZZ_{\leq 0}$, and let
$F \in M^\mer_{2-\ell}(0)$, with singularities in the set $S$ of the
form $\Gm\, \PtY$. Let $\mu_\infty\in \ZZ$ and $\mu_\z\leq -1$ for
$\z\in \PtY$ be as in~\eqref{frexp}. For each $N \geq
\max(1,1-\nobreak\mu_\infty)$ there is a neighborhood $V_{0,N}$ of
$0$ in~$\CC$ with the properties in Lemma~\ref{lem-V0N}, and on
$V_{0,N}$ a holomorphic family $r\mapsto H_{N,r}$ of
$(\ell+\nobreak r)$-harmonic modular forms in $H^\mer_{\ell+r}(r)$
such that $\x_{\ell+r} H_{N,r} = \bar\eta^{-2r}F$, uniquely
determined by having near $\infty$ a Fourier expansion of the form
\bad\label{four-norm}
H_{N,r}(z) &\=\sum_{\nu\leq -\max(N,\mu_\infty)} a_{-\nu}^\infty(r)\,
\Bigl(
-4\pi\bigl(\nu+\tfrac r{12}\bigr)\Bigr)^{\ell+r-1}\, \, q^{\nu+r/12}\\
&\qquad\qquad\hbox{} \cdot
\Gf\Bigl(1-\ell-r,-4\pi\bigl(\nu+\tfrac r{12}\bigr)y\Bigr)\\
&\quad\hbox{}+ \sum_{\nu=1-N}^{-\mu_\infty} a_{-\nu}^\infty(r) \,
q^{\nu+r/12}\, M_{\ell+r}\Bigl(4\pi\bigl(\nu+\tfrac r
{12}\bigr);y\Bigr)\\
&\quad\hbox{} + \sum_{\nu\geq 1-m_\ell} b_{N,\nu}^\infty(r)\,
q^{\nu+r/12}\,. \ead
The functions $b_{N,\nu}$ are holomorphic on $V_{0,N}$.

The holomorphic functions $a_\nu$ on~$\CC$ depend on $F$
by~\eqref{frexp}. The integer $m_\ell\geq 0$ is defined
in~\S\ref{sect-frd}. See \eqref{Mdef} for the functions~$M_{\ell+r}$.
\end{prop}
We note that there may be an overlap in the ranges of the
variable~$\nu$ in the sums in~\eqref{four-norm}, and that the sum
over $1-N\leq \nu \leq -\mu_\infty$ may be empty.
\smallskip

From this point on we use this normalization of $H_{N,r}$ and the
functions $b_{N,\nu}^\infty$. The functions $b_{N,\nu}^\infty$ are
not known explicitly. The choice of $M_{\ell,\nu}$ is not canonical.
So this normalization is non-canonical as well.

Since we deal with real-analytic functions on $\uhp\setminus S$, the
expansion near $\infty$ determines the family completely. At points
$\z\in \PtY$ we have expansions like in \eqref{ztexph}, with terms
that are holomorphic on~$V_{0,N}$.

\subsection{Extension}\label{sect-ext}
The normalization in Proposition~\ref{prop-hol-norm} is convenient for
the comparison of $H_{N,r}$ and $H_{N+1,r}$. The difference
$H_{N+1,r}-H_{N,r}$ is a holomorphic family of holomorphic modular
forms on $V_{0,N}$:
\begin{lem}\label{lem-diff}Let $N \geq \max(1,1-\nobreak\mu_\infty)$.
If $N<\max(m_\ell,\mu_\infty)$ then $H_{N+1,r}=H_{N,r}$ for
$r\in V_{0,N}$. If $N\geq \max(m_\ell,\mu_\infty) $, then we have for
$r\in V_{0,N}$:
\be\label{Nstep}
 H_{N+1,r} \= H_{N,r} - a_{-N}^\infty(r) \, \Bigl (4\pi\bigl(N-\tfrac
r{12}\bigr)\Bigr)^{\ell+r-1}\, \Gf \Bigl( 1-\ell-r,4\pi\bigl(N-\tfrac
r{12}\bigr)\Bigr)\, j_{\ell,N,r}\,. \ee
\end{lem}
\begin{proof}The difference
\begin{align*}
&H_{N+1,r}(z) - H_{N,r}(z) \= \sum_{\nu \geq 1-m_\ell} \bigl(
b^\infty_{N+1,\nu}(r) - b^\infty_{N,\nu}(r)\bigr)\, q^{\nu+r/12}\\
&\hbox{} + \begin{cases}
a_{-N}^\infty(r) \, q^{-N+r/12}\,\Bigl(M_{\ell+r}\bigl(
4\pi\bigl(-N+\tfrac r{12}\bigr);y\bigr) &\\
\quad\hbox{}- \bigl( -4\pi\bigl( -N+\tfrac r{12}\bigr)
\bigr)^{\ell+r-1}\, \Gf\bigl(1-\ell-r,-4\pi\bigl(-N+\tfrac
r{12}\bigr)y\bigr)
\Bigr) &\text{ if }N\geq\mu_\infty\,,
\\
0&\text{ if }N<\mu_\infty
\end{cases}
\end{align*}
is a holomorphic family on $V_{0,N}$ of holomorphic modular forms. If
$N<m_\ell$ or if $N<\mu_\infty$, then it has non-zero Fourier terms
only of order $\nu\geq 1-m_\ell$, hence it vanishes. If
$N\geq m_\ell$ and $N\geq \mu_\infty$, then a computation shows that
the starting term in the Fourier expansion is equal to
\[- \Bigl( -4\pi\bigl( -N+\tfrac r{12}\bigr)
\Bigr)^{\ell+r-1}\, \Gf\Bigl(1-\ell-r,-4\pi\bigl(-N+\tfrac
r{12}\bigr)\Bigr) \, a_{-N}^\infty(r) \, q^{-N+r/12}\,, \]
and the other terms have order $\nu \geq 1-m_\ell$. So
$H_{N+1,r}-H_{N,r}$ is equal to the multiple of $j_{\ell,N,r}$
indicated in the lemma.

\end{proof}
\begin{lem}\label{lem-ext}The family $r\mapsto H_{N,r}$ in
Proposition~\ref{prop-hol-norm} extends as a holomorphic family on
$\CC\setminus[12M,\infty)$, where $M=\max(N,m_\ell,\mu_\infty)$.
\end{lem}
\begin{proof}All $H_{N,r}$ with $N<\max(m_\ell,\mu_\infty)$ have a
holomorphic extension to $V_{0,N_1}$, with
$N_1=\max(m_\ell,\mu_\infty)$. The function
$w\mapsto w^{\ell+r-1}\,\Gf(1-\nobreak \ell-\nobreak r,w)$ can be
extended as a single-valued holomorphic function on
$\CC\setminus(-\infty,0]$, but not further for general values of
$\ell+r$. Lemma~\ref{lem-diff} implies that for $N\geq N_1$ the
family $H_{N,r}$ has a holomorphic extension to
$V_{0,N+1}\setminus [12 N,\infty)$. Applying this successively, we
get for $N\geq N_1$ the holomorphic extension of $H_{N,r}$
to~$\CC\setminus[12N,\infty)$.
\end{proof}

\begin{thm}\label{thm-main}Let $F$ be an antiholomorphic form in $\bar
M_{2-\ell}^\mer\bigl(\SL_2(\ZZ),v_0\bigr)$ for the trivial multiplier
system $v_0=1$, with weight $\ell\in 2\ZZ$. Let $M\in \ZZ$ be such
that $M>-\mu_\infty$, where
$F(z) = \sum_{\nu\geq \mu_\infty} a_\nu \, \bar q^\nu$ is the Fourier
expansion of $F$ near~$\infty$. Then there is a holomorphic family
$r\mapsto \ha_{M,r}$ on $\CC\setminus [12M,\infty)$ of
$(\ell+\nobreak r)$-harmonic modular forms satisfying $\x_{\ell+r}
\ha_{M,r}=\bar\eta^{-2r}\, F$ for all
$r\in \CC\setminus[12M,\infty)$.

The family $\ha_{M,r}$ can be chosen uniquely by prescribing a Fourier
expansion of the form
\badl{haFe} \ha_{M,r}(z) &\= \sum_{\nu\leq -\max(M,\mu_\infty)}
a_{-\nu}^\infty(r)\, \Bigl(
-4\pi\bigl(\nu+\tfrac r{12}\bigr)\Bigr)^{\ell+r-1}\, \, q^{\nu+r/12}\\
&\qquad\qquad\hbox{} \cdot
\Gf\Bigl(1-\ell-r,-4\pi\bigl(\nu+\tfrac r{12}\bigr)y\Bigr)\\
&\quad\hbox{}+ \sum_{\nu=1-M}^{-\mu_\infty} a_{-\nu}^\infty(r) \,
q^{\nu+r/12}\, M_{\ell+r}\Bigl(4\pi\bigl(\nu+\tfrac r
{12}\bigr);y\Bigr)\\
&\quad\hbox{} + \sum_{\nu\geq 1-m_\ell} b_{M,\nu}^\infty(r)\,
q^{\nu+r/12}\,. \eadl
The function $M_{\ell+r}$ is defined in~\eqref{Mdef}. The coefficients
$a_{-\nu}^\infty$ are holomorphic functions on~$\CC$ occurring in the
Fourier expansion $(\bar\eta^{-2r}\, F)(z) =
\sum_{\nu\geq \mu_\infty}a_\nu^\infty(r) \, \bar q^{\nu-r/12}$. The
quantity $m_\ell\in \ZZ$ is defined by
$\ell+12 m_\ell\in \{0,4,6,8,10,14\}$.
\end{thm}
The holomorphic functions $b_{M,\nu}$ on $\CC\setminus[12M,\infty)$
are not known explicitly. The middle sum in~\eqref{haFe} may be
empty. We note that $a_{\mu_\infty}^\infty(r) = a_{\mu_\infty}$,
which we can assume to be non-zero. Then the top term in the central
sum is non-zero.
\begin{proof}
We denote by $\tilde F\in \bar M_{2-\tilde\ell}^\mer(v_0)$ the
antiholomorphic modular form in the theorem, and will apply the
earlier result to
\[ F \= \tilde F\, \bar\Dt^{-p}\,\in\, \bar M^\mer_{2-\ell}(v_0)\,,\]
with $p\in \ZZ$ not yet fixed. Denoting the quantities related
to~$\tilde F$ by a tilde we have
\begin{align*}
\ell&\= \tilde\ell+12p\,,&\qquad m_\ell&\= m_{\tilde \ell}-p\,,\\
\mu_\infty &\=\tilde\mu_\infty-p,\,,& a_\nu(r) &\= \tilde
a_{\nu+p}(r+12p)\,.
\end{align*}
We take $N=M-p$ and choose
\[ p \leq \min\Bigl(
-\tfrac{\tilde\ell}{12},M-1,\tfrac{M-1+\tilde\mu_\infty}2\Bigr)\,.\]
Then $\ell\leq 0$ and $N\geq \max(1,1-\nobreak\mu_\infty)$. We apply
Proposition~\ref{prop-hol-norm} and Lemma~\ref{lem-ext} to $F$, and
take $\ha_{M,r}= H_{N,r-12p}$. Then $\x_{\tilde \ell+r} \ha_{M,r} \=
\x_{\ell+r-12p} H_{N,r-12p} = \bar\eta^{-2r}\,\tilde F$ for $r\in
\CC\setminus[12M,\infty)$.
\end{proof}

\subsection{Remarks}
\subsubsection{Comparison: use of Poincar\'e series and use of
perturbation theory}The existence of a harmonic lift of a single
antiholomorphic modular form can be proved with Poincar\'e series in
the case of a real weight and a unitary multiplier system. If the
weight is larger than~$2$ the Poincar\'e series converge absolutely,
and the construction gives an explicit expression for the lift (with
Kloosterman sums and Bessel functions). Outside the region of
absolute convergence analytic extension is needed anyhow, and
requires a careful analysis of the properties of this continuation.
The approach in \S\ref{sect-ramf} uses a more general result, and
works generally. I do not know another method that allows the
handling of complex weights.

\subsubsection{Use of Hodge theory}The approach in \S3 of~\cite{BrFu}
is not restricted to the case of
$\bar M^!_{2-p}\bigl(\SL_2(\ZZ),v_1\bigr)$. Jan Bruinier remarks that
singularities at points of $\uhp$ can be accommodated in the divisor
$D$ used in the proof of Theorem~3.7 in~\cite{BrFu}, and that one may
be able to handle multiplier systems $v_r$ with rational values
of~$r$.

\subsubsection{Existence only}Theorem~\ref{thm-main} is an existence
result. It gives an overview of harmonic lifts and organizes them in
families. It does not give explicit knowledge of the lifts.

Sometimes we know explicitly a harmonic function by other means, and
may be able to identify it as a member of a family. See
\S\ref{sect-etap} for some examples.

\subsubsection{Generalization}\label{generalization}Theorem~\ref{thm-main}
is stated only for the discrete group $\SL_2(\ZZ)$, since for that
case I have checked the details. I expect that a similar theorem can
be proved for any cofinite discrete subgroup of $\SL_2(\RR)$ with
cusps. For cocompact groups generalization seems much harder.

For cofinite groups $\Gm$ with cusps the group of multiplier systems
is a commutative complex Lie group, with finite dimension; its
dimension is $1$ for $\SL_2(\ZZ)$. The parameter $r$ in this paper is
essentially an element of the Lie algebra of the group of multiplier
systems. The parameter $\ph$ used in~\cite{Br94} can be viewed as
running through the Lie algebra of the group of multiplier systems.
The results in \S\ref{sect-ramf} probably go through with open sets
$V_{0,N}$ in that Lie algebra as parameter space. To transform
meromorphic families into holomorphic families with the method of
\S\ref{sect-hfhf} is probably very hard if the dimension of the
parameter space is larger than~$1$. For that purpose I think it might
be wise to work with one-dimensional subvarieties of the parameter
space.

\subsection{Mock modular forms}\label{sect-mock}
In the Fourier expansion~\eqref{four-norm} of the normalized family
$\ha_{M,r}$ it seems natural to put
\bad \co_{M,r}(z) &\= \sum_{\nu\leq -\max(M,\mu_\infty)}
a_{-\nu}^\infty(r)\, \Bigl(
-4\pi\bigl(\nu+\tfrac r{12}\bigr)\Bigr)^{\ell+r-1}\, q^{\nu+r/12}\\
&\qquad\hbox{} \cdot
\Gf\Bigl(1-\ell-r,-4\pi\bigl(\nu+\tfrac r{12}\bigr)y\Bigr)\\
&\qquad\hbox{}
+ \sum_{\nu=1-M}^{-\mu_\infty} a_{-\nu}^\infty(r) \,
M_{\ell+r}\Bigl(4\pi\bigl(\nu+\tfrac r{12}\bigr);y\Bigr)\,
q^{\nu+r/12} \ead
as the part of the expansion arising from $\bar\eta^{-2r}\, F$, and
the remaining part
\be\label{Mo-def} \mo_{M,r}(z) \= \sum_{\nu \geq 1-m_\ell}
b_{N,\nu}^\infty(r)\, q^{\nu+r/12}\ee
as the corresponding family of \emph{mock modular forms}. This
splitting depends on the choice of the basis vector
$M_{\ell+r}\bigl( \cdot;\cdot)$ in the $(\ell+\nobreak r)$-harmonic
terms with $\nu\geq 1-m_\ell$. Another choice leads to another
splitting.

If the set $S$ of singularities in~$\uhp$ of $F$ is non-empty, the
series in \eqref{Mo-def} for $\mo_{M,r}(z)$ defines a holomorphic
function only on the region $\im z>y_S$ where $y_S$ is the maximum
value om $\im \z$ as $\z$ runs through~$S$. It seems unknown whether
the functions $\mo_{N,r}$ and $\co_{N,r}$ have an analytic extension
to a larger region in~$\uhp$.

The expansion of $\ha_{M,r}$ at $\z$ in the system of representatives
$\PtY$ of $\Gm\backslash S$ gives rise to a splitting $\ha_{M,r}=M+C$
on a pointed neighborhood of $\z$, and seems not to have a relation
to the splitting $\ha_{M,r}=\mo_{M,r}+\co_{M,r}$.

My conclusion is that the concept of mock modular forms is still
unclear in the generality of families of modular forms considered in
this note.

\section{Harmonic lift of eta-powers}\label{sect-etap}
As an example we look at $\bar \eta^{-2r}$ for $r$ near to~$0$
in~$\CC$. Theorem~\ref{thm-main}, with the choice
$1\in \bar M_{2-2}\bigl(\SL_2(\ZZ),v_0\bigr)$ provides us with the
family $r\mapsto \ha_r:=\ha_{1,r}$ on $\CC\setminus[12,\infty)$ of
$(r+\nobreak 2)$-harmonic lifts of~$\bar \eta^{-2r}$. We identify it
with known harmonic lifts for certain values of~$r$.\smallskip

We note that
\be \bar\eta^{-2r} \= \sum_{\nu\geq 0} p_\nu(-r)\, \bar
q^{\nu-r/12}\,, \ee
with polynomials $p_\nu$ of degree~$\nu$ with rational coefficients. A
first order expansion of $e^{2r\log\eta} $ at $r=0$ shows, with use
of~\eqref{logeta}
\be p_0=1\,,\quad\text{ and for }\nu\geq 1:\quad p_\nu(0)=0\text{ and
}p_\nu'(0)=-2\, \s_{-1}(\nu)\,.\ee

Theorem~\ref{thm-main} gives the following Fourier expansion:
\bad \ha_r(z) &\= \sum_{\nu\leq -1} p_{-\nu}(-r) \,
\Bigl(-4\pi\bigl(\nu+\tfrac r{12}\bigr)\Bigr)^{r+1}\, q^{\nu+r/12}\,
\Gf\Bigl( -1-r,-4\pi\bigl(\nu+\tfrac r{12}\bigr)y\Bigr)
\\
&\qquad\hbox{}
+ \Bigl( M_{2+r}\bigl( \tfrac{\pi r}3;y\bigr) + b_0(r) \Bigr)\,
q^{r/12}
+ \sum_{\nu\geq 1} b_\nu(r) \, q^{\nu+r/12}\,. \ead
The holomorphic functions $b_\nu$ on $\CC\setminus[12,\infty)$ are
unknown.

For three values of $r$ we mention constructions of
$(r+\nobreak 2)$-harmonic modular lifts of $\bar\eta^{-2r}$. If
$r\in (-12,12)$ the sole term of $\ha_r$ that q is not exponentially
decreasing at~$\infty$ is the term of order $\frac r{12}$. If the
other lift also has this term as the only non-decreasing one, that
lift coincides with~$\ha_r$.
\medskip\par\noindent
\emph{Lift of $1$. }A well known $2$-harmonic lift of $1=\bar\eta^0$
is the non-holomorphic Eisenstein series
\be E_2^{\mathrm{nh}}(z)= y^{-1}-\frac \pi 3 + 8\pi \sum_{\nu \geq
1}\s_1(\nu)\, q^\nu \,.\ee

We have $M_2(0;y)= y^{-1}-1$. The terms with $q^\nu$, $\nu\geq 1$ are
quickly decreasing. We conclude that $\ha_0=E_2^{\mathrm{nh}}(z)$,
and find
\be b_0(0)\=1-\frac \pi3\,,\qquad b_\nu(0)\= 8\pi\,
\s_1(\nu)\quad\text{for }\nu\geq 1\,. \ee
So we have identified the value of the family at $r=0$ with a known
$2$-harmonic modular form.

In this case we can proceed a bit further. The computation used
in~\S6.4 of~\cite{BrDi} to produce an explicit example of a second
order Maass form can be modified to get information on the derivative
$\frac d{dr}\ha_r\bigr|_{r=0}$. In this way one can arrive at the
following result:
\be b_\nu'(0) \=
\begin{cases} -16\pi \sum_{\mu=1}^{\nu-1}\s_{-1}(\mu)\,\s_1(\nu-\mu)
-8\pi \sum_{d|\nu}\frac \nu d\log\frac {d^2}\nu
\\\quad\hbox{}
- 8\pi\,\bigl(1+\gm-\log 4\pi\bigr)\, \s_1(\nu) +\frac{2\pi}3\,
\s_{-1}(\nu)\,,&\text{ if }\nu\geq 1\,,\smallskip\\
-1+\frac\pi 3\bigl(2\gm-\log 4\bigr)-\frac 4\pi\z'(2)&\text{ if
}\nu=0\,,
\end{cases}
\ee
where $\gm$ denotes Euler's constant.

\medskip\par\noindent
\emph{Lift of $\eta^3$. }This case can be related to an example of a
mock modular form in~\cite{Zw02}. The unary theta function $g_{a,b}$
in Proposition~1.15 in~\cite{Zw02} with $a=b=\frac12$ gives
$g_{1/2,1/2}=\eta^3$. (This follows from the transformation behavior
of $g_{a,b}$ and inspection of its Fourier expansion.) The completed
Lerch sum $\tilde \mu$ in Theorem~1.11 in~\cite{Zw02} gives a
$\frac12$-harmonic lift
\be z\mapsto \frac{\sqrt 2}{3i}\, \Bigl(\tilde \mu\bigl(
\tfrac12,\tfrac12;z)+ \tilde \mu \bigl(\tfrac z2,\tfrac z2;z) +
\tilde \mu\bigl( \tfrac{z+1}2,\tfrac{z+1}2;z\bigr)\Bigr)\ee
of $\bar \eta^3$. Since the term with $q^{-1/8}$ is the sole
increasing term (as $y\rightarrow\infty)$, this lift is equal to
$\ha_{-3/2}(z)$.
\medskip\par\noindent
\emph{Lift of $\eta^4$. }The fourth power $\eta^4$ spans the space of
holomorphic cusp forms for the commutator subgroup
$\Gcom=[\SL_2(\ZZ),\SL_2(\ZZ)]$.

The holomorphic function $H$ on~$\uhp$ given by
\be H(\tau) \= -2\pi i\int_\infty^z \eta^4(\tau)\, d\tau \ee
has the transformation behavior $H(\gm \tau) = H(\tau) + \ld(\gm)$ for
some group homomorphism $\ld:\Gcom\rightarrow\CC$. The function
$C(z) = \frac{-1}{4\pi}\,\bar H$ satisfies $\x_0\, C = \bar \eta^4$.
In \S4.3.1 in~\cite{BrDi} we see that there can be found a linear
combination $M$ of the holomorphic functions $H$ and
$z\mapsto \z\bigl(H(z) \bigr)$, where $\z$ is the Weierstrass
zeta-function for an appropriate lattice, such that $M+C$ is a
$\Gcom$-invariant harmonic lift of $\bar \eta^4$. The average
$\sum_{n\bmod 6} e^{\pi i n/3}\,(M+\nobreak C) |_0T^n$ has the
desired transformation behavior under $\SL_2(\ZZ)$. Inspection of the
growth of the Fourier terms of $M(z) + C(z)$ as
$\im z\rightarrow\infty$ shows that $M+C$ is equal to~$\ha_{-4}$.


\ifcitenumber
\newcommand\bibit[4]{
\bibitem {#1}#2: {\em #3;\/ } #4}
\else
\newcommand\bibit[4]{
\bibitem[#1] {#1}#2: {\em #3;\/ } #4}
\fi
\raggedright

\end{document}